\documentclass[12pt]{article}
\usepackage[margin=1in]{geometry}
\usepackage[utf8]{inputenc}
\usepackage[english]{babel}
\usepackage{caption}
\usepackage{subcaption}

\usepackage[dvipsnames]{xcolor}

\usepackage{amsmath, amssymb, amsthm,mathtools,dsfont}
\usepackage{bbm}
\PassOptionsToPackage{numbers,square,sort&compress}{natbib}
\usepackage{natbib}
\usepackage{blkarray}
\usepackage{cancel}
\usepackage{enumitem}
\usepackage{euscript}
\usepackage{float}
\usepackage{bm}
\usepackage{graphicx}
\usepackage{mathrsfs}
\usepackage{microtype}
\usepackage{ragged2e}
\usepackage{sectsty}
\usepackage{thmtools}
\usepackage{tikz}
\usepackage[Symbolsmallscale]{upgreek}

\usepackage{microtype}
\usepackage{graphicx}
\usepackage[pdfborder={0 0 0}]{hyperref}
\hypersetup{
  urlcolor = black,
  pdfauthor = {Shunan Sheng, Bohan Wu, Alberto Gonzalez-Sanz, Marcel Nutz},
  pdfkeywords = {Mean-Field, Variational Inference, Optimal Transport},
  pdftitle = {Stability of Mean-Field Variational Inference},
  pdfsubject = {Stability of Mean-Field Variational Inference},
  pdfpagemode = UseNone
}

\input{mathsymbols}
\usepackage{algorithm}
\usepackage{algpseudocode}

\usepackage{fancyhdr}

\renewenvironment{thebibliography}[1]{%
\begin{oldthebibliography}{#1}%
\setlength{\baselineskip}{.85em}
\linespread{.9}
\small
\setlength{\parskip}{.07ex}%
\setlength{\itemsep}{.07em}%
}%
{%
\end{oldthebibliography}%
}

\theoremstyle{plain}

\begin{document}
\title{Stability of Mean-Field Variational Inference}
\author{
  Shunan Sheng%
  \thanks{
  Columbia University, Department of Statistics, ss6574@columbia.edu.}  \and
  Bohan Wu%
  \thanks{ Columbia University, Department of Statistics, bw2766@columbia.edu.}  \and
 Alberto Gonz\'alez-Sanz%
  \thanks{
Columbia University, Department of Statistics, ag4855@columbia.edu. }
 \and  Marcel Nutz%
  \thanks{
  Columbia University, Departments of Statistics and Mathematics, mnutz@columbia.edu. Research supported by NSF Grants DMS-2106056, DMS-2407074.}
  }
\date{\today}
\maketitle 

\begin{abstract}
Mean-field variational inference (MFVI) is a widely used method for approximating high-dimensional probability distributions by product measures.  This paper studies the stability properties of the mean-field approximation when the target distribution varies within the class of strongly log-concave measures. We establish dimension-free Lipschitz continuity of the MFVI optimizer with respect to the target distribution, measured in the 2-Wasserstein distance, with Lipschitz constant inversely proportional to the log-concavity parameter. Under additional regularity conditions, we further show that the MFVI optimizer depends differentiably on the target potential and characterize the derivative by a partial differential equation. Methodologically, we follow a novel approach to MFVI via linearized optimal transport: the non-convex MFVI problem is lifted to a convex optimization over transport maps with a fixed base measure, enabling the use of calculus of variations and functional analysis. We discuss several applications of our results to robust Bayesian inference and empirical Bayes, including a quantitative Bernstein--von Mises theorem for MFVI, as well as to distributed stochastic control.
\end{abstract}

\vspace{.9em}

{\small

}

\section{Introduction}

Given a probability measure $\pi$ on $\R^d$ with potential $V: \R^d \mapsto \R\cup \{+\infty\}$, i.e.,
\begin{equation*}
\pi(\rd x)  = Z^{-1}e^{-V(x)}\dd x 
\end{equation*}
where $Z$ denotes the normalizing constant, the mean-field variational inference (MFVI) problem is to find the best approximation of $\pi$ within the set $ \cP(\R)^{\otimes d}$ of product measures in the sense of Kullback--Leibler (KL) divergence, 
\begin{equation}\label{eq: MFVI}
   \nu^* \in \argmin_{\nu \in \cP(\R)^{\otimes d}} \kl{\nu}{\pi}. \tag{MFVI}
\end{equation}
While $\kl{\cdot}{\pi}$ is a strictly convex functional, the set $\cP(\R)^{\otimes d}$ of product measures is not convex in the usual sense. Indeed, \eqref{eq: MFVI} generally admits more than one optimizer. Lacker, Mukherjee and Yeung \cite{Lacker2024} resolved this issue when the target measure $\pi$ is $\alpha$-log-concave, i.e., the potential $V$ is $\alpha$-strongly-convex. Using that  $\cP(\R)^{\otimes d}$ is \emph{displacement} convex (geodesically convex in the optimal transport geometry), their seminal work shows that \eqref{eq: MFVI} then admits a unique optimizer $\nu^*$ with several desirable properties (see also \cref{th:Lacker2024} below), in particular the approximation bound
\begin{equation*}
 \cW_2^2(\nu^*, \pi) \leq \frac{2}{\alpha^3} \sum_{1 \leq i < j \leq d} \int |\partial_{ij} V(x)|^2 \nu^*(\rd x)
\end{equation*}
where $\cW_2$ denotes the $2$-Wasserstein metric. 
The present paper establishes \emph{stability} of the optimizer~$\nu^*$ under perturbations of the target $\pi$ within the class of log-concave target measures. 
To the best of our knowledge, these are the first quantitative stability results for MFVI.

MFVI has gained popularity in large-scale Bayesian applications~\citep{jordan1999introduction, Wainwright2008, Blei2017} due to its computational tractability, especially when sampling from the posterior distribution~$\pi$ is required for downstream tasks. Unlike traditional Markov chain Monte Carlo (MCMC) methods~\citep{hastings1970monte, gelfand1990sampling, Robert2004} which can be prohibitive in high dimensions, MFVI offers a computationally scalable alternative, especially when the posterior admits a complex graphical structure~\citep{Wainwright2008}. Although MFVI relies on sampling from an approximate posterior, these samples often provide sufficiently accurate inference for practical applications such as genomics~\citep{Carbonetto2012,Wang2020EBVI-VS}, generative modeling~\citep{KingmaWelling2019}, or language modeling~\citep{blei2003latent}. 
Stability plays a crucial role in the area of \textit{robust Bayesian analysis}, where the goal is to quantify the sensitivity of posterior inference under perturbations of the model or prior~\citep{Berger1994, Gustafson2002, Giordano2018}. As MFVI is increasingly used in place of exact posterior inference for large-scale and complex models, understanding its stability is essential for evaluating the robustness of statistical decisions and analysis made based on the variational approximation.

The assumption of log-concavity is standard in the literature on log-concave sampling \citep{Chewi2023} which is the foundation for several recent theoretical developments in variational inference \citep{Jiang2023, Lacker2024, arnese2024convergence}. The classical application of MFVI in statistical physics---here more commonly known as the \emph{mean-field approximation}---often assumes that the target measure $\pi$ is a discrete Gibbs measure supported on the discrete hypercube $\{-1,1\}^d$~\citep{Parisi1980, Opper2001,Basak2017}. A parallel  line of work has developed quantitative error bounds for mean-field approximations using a “gradient complexity” framework, both for models defined on the discrete hypercube~\citep{Chatterjee2016, Basak2017, Eldan2020, Augeri2021} and for those supported on general compact spaces~\citep{Yan2020, Mukherjee2021, Augeri2020}. Although the compact support setting is of independent interest, we do not consider it in this paper as the MFVI optimizer is often not unique in that setting (as noted in~\citep{Mukherjee2021}), precluding stability in the sense shown here.

Next, we discuss our methodology and main results---in this order, as the statement of the second main result uses an object best presented as part of the methodology.

\paragraph{Methodology} We follow the novel approach to MFVI based on \emph{linearized} optimal transport~\citep{wang2013linear,jiang2023algorithms}. Fixing the $d$-dimensional standard Gaussian $\rho$ as reference measure, we can parametrize $\nu\in\cP_2(\R^d)$ by the unique optimal transport map $T$ from $\rho$ to $\nu$. Optimal transport maps (for the quadratic transport cost) are precisely the gradients of convex functions on $\R^d$, and the push-forward $\nu$ of $\rho$ under the map $T$ is a product measure if and only if $T(u_1,\ldots,u_d)=(T_1(u_1),\ldots, T_d(u_d))$ where each $T_i$ is the gradient of a convex function on~$\R$. As a result, we can recast the MFVI problem as an optimization over such maps which naturally \emph{form a convex set}. The KL divergence $\kl{\cdot}{\pi}$ translates to the strictly convex functional
\begin{equation}\label{eq: F_Vintro}
    \cF_V( T_1,\ldots, T_d)=  
        - \sum_{i=1}^d \int_\R \log (T_i'(u_i)) \rho_1(\rd u_i) +\int V( T_1(u_1),\ldots, T_d(u_d)) \rho(\rd u)
\end{equation}
up to a constant (see \cref{lemma:lifted KL} for details). In summary, the non-convex MFVI problem is lifted to a convex problem over transport maps. While novel in the present context, this seems to be a natural and powerful approach to MFVI. For instance, the uniqueness of the MFVI optimizer (already proved in \cite{Lacker2024}) is immediate from this perspective. All our results are derived using the ``lifted'' MFVI problem whose convexity enables powerful tools of calculus of variations and functional analysis.

The use of optimal transport theory is certainly not unprecedented in studying the MFVI problem. As mentioned, \cite{Lacker2024} used the Wasserstein geometry~\citep{otto2001geometry,lott2009ricci} to show uniqueness; this is a curved geometry in contrast to the flat geometry of our lifted problem. 
More closely related is the recent work \cite{jiang2023algorithms}, developing a polynomial-time algorithm for solving MFVI. The authors also translate the variational problem into an optimization over transport maps. Specifically, they fix a finite-dimensional set of transport maps and compute an approximate solution of the MFVI problem using gradient descent on that set. In the present work, we introduce an infinite-dimensional optimization problem on a Gaussian Sobolev space and show that it is equivalent to the original variational problem.
Another related paper is~\cite{arnese2024convergence},  showing the convergence of the Coordinate Ascent Variational Inference (CAVI) algorithm, an algorithm that is commonly used to find the optimizer. The analysis in~\cite{arnese2024convergence} focuses on the tangent space at the MFVI optimizer or the CAVI iterates, and does not apply when the target measure is itself varying as in our problem. 
The identification of probability measures with their corresponding transport maps also plays a central role in rank-based inference. When the base measure is the uniform distribution on $[0,1]$, the optimal transport map coincides with the quantile function of the associated probability measure. In higher dimensions, \cite{Chernozhukov2017, HallinBarrio2021, ghosal2022multivariate, deb2023multivariate} and others used optimal transport maps to define multivariate analogues of quantile functions and conduct rank-based inference. 

\paragraph{Results} We study the stability of the MFVI procedure: how sensitive is~$\nu^*$ to perturbations of the potential $V$ of the target~$\pi$?  Our first main result (\cref{thm:Lip}) shows that when $\pi, \tilde \pi$ are $\alpha$-log-concave, the corresponding MFVI optimizers $\nu^*,\tilde{\nu}^*$  satisfy the dimension-free Lipschitz stability bound
\begin{equation} \label{result-sketch}
 \cW_2(\nu^*,\tilde{\nu}^*)\leq \frac{\| \nabla V-\nabla \tilde V\|_{L^2(\nu^*)} }{\alpha}.
\end{equation}
This can be seen as a mean-field analogue of classical transportation--information inequalities (e.g., \citep{Djellout2004, Guillin2009TII1, Augeri2021}) with $\| \nabla V-\nabla \tilde V\|_{L^2(\nu^*)}$ taking the place of the relative Fisher information between $\pi$ and $\tilde \pi$. We also provide a Lipschitz result for the optimal KL divergence achieved by~$\nu^*$, or equivalently the maximal evidence lower bound (ELBO); see \cref{thm:reward lip}. Recognizing that the right-hand side of~\eqref{result-sketch} is not directly accessible if $\nu^*$ is unknown, \cref{coro:Lip} complements \cref{thm:Lip} by giving upper bounds that are explicit in terms of the given data.

The bound~\eqref{result-sketch} is based on the variational first-order condition of optimality for the lifted MFVI problem: at the optimizer, the first variation of the functional~\eqref{eq: F_Vintro} in any admissible direction vanishes. We prove~\eqref{result-sketch} by using this fact at both optimizers, $\nu^*$ and $\tilde{\nu}^*$. 

Our second main result (\cref{th: Differentiable}) establishes the differentiability of the MFVI optimizer as a function of the potential $V$ and characterizes the derivative. To that end, we use the lifted problem where derivatives can be defined in the usual sense thanks to the flat geometry. Specifically, consider a family of potentials $\{V_\theta\}_{\theta \in \Theta}$ that are $\cC^2$, $\alpha$-convex and $\beta$-smooth, and assume that the gradient $\nabla V_\theta$ is Fr\'echet differentiable in $\theta$ (in a suitable $L^2$ space) with derivative denoted by $\partial_\theta \nabla V_{\theta}$. We show that the optimizer $T^\theta$ of the lifted functional~\eqref{eq: F_Vintro} is differentiable in $\theta$ and that the derivative $\partial_\theta T^{\theta}$ can be characterized as the solution to a partial differential equation. %

For applications, identifying the derivative $\partial_\theta T^{\theta}$ is useful to quantify the sensitivity of the MFVI optimizer under perturbations of the target measure. This is relevant for analyzing the robustness of downstream statistical procedures that rely on the variational approximation~\citep{yao2018using,Zhang2024}. 

The proof of \cref{th: Differentiable} builds on a second-order analysis of the functional~\eqref{eq: F_Vintro} whose form precludes a straightforward application of the implicit function theorem. Instead, we construct the derivative $\partial_\theta T^{\theta}$ via the Lax--Milgram theorem in the space $(\cH^1(\rho_1))^d$, where~$\rho_1$ is the one-dimensional standard Gaussian measure. In this argument, $\alpha$-convexity of the potentials is exploited to guarantee coercivity of the bilinear form. Another key ingredient for the proof of \cref{th: Differentiable} is a novel $L^p$ estimate for the difference $T^\theta - T^{\theta_0}$ where $p > 2$ (\cref{{pr: Lp estimate}}). As a by-product, this estimate also leads to a stability bound for $\nu^*$ in the $p$-Wasserstein metric for $p > 2$; see \cref{prop: Wp Lipschitz}.

Several applications in statistical inference and stochastic control are presented. On the strength of our theoretical results, we quantify the \emph{robustness of variational Bayesian inference} in several situations including high-dimensional Bayesian linear regression and the prior swapping problem (\cref{sec:variational bayes}). In our likely most impactful statistical application, we derive a novel \emph{quantitative Bernstein--von Mises (BvM) theorem} that bounds the Wasserstein distance between the MFVI optimizer and a mean-field analogue of the Laplace approximation of the posterior (\cref{sec: quantitative BvM}). This is the  first quantitative BvM theorem for MFVI in a high-dimensional setting. Separately, our results on MFVI imply a quantitative stability result for the optimal value of \emph{distributed stochastic control} problems (\cref{sec: distributed sc}). 
Last but not least, the established stability of the MFVI optimizer is a crucial auxiliary result in forthcoming work on \emph{structured variational inference} where a given target measure is approximated by distributions following a prescribed dependence structure rather than having independent marginals as in the present mean-field case. %

\paragraph{Organization} 
The remainder of the paper is organized as follows.  \Cref{se: notation and prelims} details notation and recalls the existence and uniqueness result for the MFVI optimizer. \Cref{sec:results} states our main results. \Cref{sec:applications} applies our theoretical results to problems in statistical inference and stochastic control. \Cref{sec:proof-Lip} contains the proofs for the Lipschitz-regularity results while \cref{sec:proof-differentiable} contains the proof of the differentiability result. \Cref{se:technicalLemmas} gathers some technical lemmas and their proofs.

\subsection{Notation and Preliminaries}\label{se: notation and prelims}
Let $\cP(\R^d)$ denote the space of probability measures on $\R^d$ and $\cP_2(\R^d)$ the subspace of measures with finite second moment. For $\nu,\tilde{\nu}\in\cP_2(\R^d)$, the 2-Wasserstein distance $\cW_2(\nu,\tilde{\nu})$ is defined via 
$
 \cW_2^2(\nu,\tilde{\nu}) = \inf \int \|x-y\|^2 \gamma(\rd x,\rd y),
$
where the infimum is taken over all couplings $\gamma$ of $(\nu,\tilde{\nu})$, i.e., all $\gamma\in \cP(\R^d\times \R^d)$ with marginals $(\nu,\tilde{\nu})$.
The Kullback--Leibler (KL) divergence between $\nu, \pi \in \cP(\R^d)$ is defined as $\kl{\nu}{\pi} = \int \log\left( \frac{\dd\nu}{\dd \pi} \right) \dd\nu$ if $\nu$ is absolutely continuous with respect to $\pi$ (denoted $\nu \ll \pi$) and $\kl{\nu}{\pi} = \infty$ otherwise. Here, $\frac{\dd\nu}{\dd \pi}$ denotes the Radon--Nikodym derivative of $\nu$ with respect to $\pi$. We denote by $\cL_d$ the Lebesgue measure on $\R^d$. When $\nu\ll\cL_d$, we use $\nu$ to denote both the measure and its density with respect to $\cL_d$. The (negative) differential entropy is then defined as $H(\nu) = \int \log\nu(x) \nu(\rd x)$ if $\nu\ll\cL_d$ and $\infty$ otherwise.

Let $\cP(\R)^{\otimes d} \subset \cP(\R^d)$ be the set of all product probability measures $\nu = \bigotimes_{i=1}^d \nu_i$. Given $\nu=\bigotimes_{i=1}^d\nu_i\in\cP(\R)^{\otimes d}$ and $j\in [d] := \{1, \dots,d\} $, we denote $\nu_{-j}= \bigotimes_{i\neq j} \nu_i$. Similarly, for a vector $x=(x_1, \dots, x_d)$ and $j\in  [d]$, we denote by $x_{-j}\in \R^{d-1}$  the vector with entries $(x_i)_{i\neq j}$. We write the integral of $f$ with respect to $\pi$ as either $\int f \dd \pi$ or $\E_{\pi}[f(X)]$, depending on the context (here $X$ is the identity function on $\R^d$). Unless otherwise specified, all integrals are taken over $\R^d$. Given a measurable function $T: \R^d\to\R^d$ and $\mu\in \cP(\R^d)$, we write $T_\sharp\mu$ for the push-forward measure $\mu\circ T^{-1}$.

We denote by $\rho_1$ the standard Gaussian distribution $\cN(0,1)$ on $\R$ and by $\rho = \bigotimes_{i = 1}^d \rho_1$ the $d$-dimensional standard Gaussian $\cN(0,\mathrm{I}_d)$. The $L^p(\rho)$ norm of a vector-valued map $f = (f_1,\ldots, f_d):\R^d \to \R^d$ is defined as $
\|f\|_{L^p(\rho)} = \left(\int \|f(u)\|^p  \rho(\rd u)\right)^{1/p}
$,
where $\|\cdot\|$ denotes the Euclidean norm on $\R^d$.  The $L^p(\rho)$ norm of a matrix-valued map $M:\R^d \to \R^{d\times d}$ is defined as 
$
\|M\|_{L^p(\rho)} = \left(\int \|M(u)\|_2^p  \rho(\rd u)\right)^{1/p}
$,
where $\|\cdot\|_2$ denotes the spectral norm on $\R^{d \times d}$. 

We write $\cH(\rho) := (\cH^1(\rho_1))^d$, where $\cH^1(\rho_1)$ is the Gaussian Sobolev space (see, e.g., \cite[Chapter~1.4]{bogachev1998gaussian}) defined as the completion of $\cC_c^\infty(\R)$ with respect to the Sobolev norm
\[
\|f\|_{\cH^1(\rho_1)} := \left( \int |f|^2 \, \dd\rho_1 + \int |f'|^2 \, \dd\rho_1 \right)^{1/2}.
\]
Note that $\cH(\rho)$ is a Hilbert space with inner product $\langle (f_1, \dots, f_d), (g_1, \dots, g_d) \rangle_{\cH(\rho)} = \sum_{i=1}^d \langle f_i, g_i \rangle_{\cH^1(\rho_1)}$. 

We recall that a map $f:\mathcal{B}_1\to \mathcal{B}_2$ between Banach spaces $(\mathcal{B}_i,\|\cdot\|_{\mathcal{B}_i}) $ is Fr\'echet differentiable at $a\in \mathcal{B}_1$ with derivative $ L  $
if
\[
\lim_{\|h\|_{\mathcal{B}_1} \to 0} \frac{\| f(a + h) - f(a) - L(h) \|_{\mathcal{B}_2}}{\|h\|_{\mathcal{B}_1}} = 0.
\] 
For a space $\cF$ of functions, $(\cF)^{d}$ denotes the product space $\{f = (f_1,\ldots, f_d): f_i \in \cF\}$. The coordinate-wise derivative of $f\in (\cF)^d$ is denoted by $f' := (f_1',\ldots, f_d')$. Generally, $C_{a_1,\ldots,a_n}$ denotes a constant depending (only) on the parameters $a_1, \ldots, a_n$.

A probability measure $\pi\in \cP(\R^d)$ is said to be \emph{$\cC^2$ and $\alpha$-log-concave} if there exists a convex function $V\in \cC^2(\R^d)$, called the \emph{potential} of $\pi$, such that $ \frac{\dd \pi}{\dd x}(x)  = \frac{e^{-V(x)}}{\int e^{-V(x)} \dd x} $ and 
\begin{equation*}
    \langle u, [\nabla^2V(x)] u\rangle \geq \alpha>0 \quad \text{for all $x,u\in\R^d$}.
\end{equation*}
The latter condition is equivalent to $\|\nabla V(x) - \nabla V(x')\| \geq \alpha \|x-x'\|$ for all $x,x' \in \R^d$. If $V$ is \emph{$\beta$-smooth,} i.e., the Hessian matrix  satisfies $\nabla^2 V\preceq \beta \mathrm{I}_d$, the measure $\pi$ is called \emph{$\beta$-log-smooth.} 

Finally, we recall the result of \cite[Theorem~1.1]{Lacker2024} ensuring existence and uniqueness for the  MFVI problem in our setting.

\begin{theorem}[Existence and uniqueness] \label{th:Lacker2024} Let $\pi\in\cP(\R^d)$ be $\cC^2$ and $\alpha$-log-concave with potential function $V:\R^d\to \R$ satisfying
\begin{equation}
    \label{eq:exponentialGrow2}
 |{V(x)}| \leq c_1 { e^{c_2 \|x\|^2}} \quad \text{for all $x\in \R^d$},
\end{equation}
for some $c_1\geq 0$ and $c_2<\frac{\alpha}{2}$. Then the MFVI problem $\inf_{\nu \in \cP(\R)^{\otimes d}} \kl{\nu}{\pi}$ has a unique solution $\nu^*=\bigotimes_{i=1}^d\nu^*_i\in\cP(\R)^{\otimes d}$, called the MFVI optimizer for target measure $\pi$. Moreover, $\nu^*$ is again $\cC^2$ and $\alpha$-log-concave. 

The measure $\nu^*$ is characterized as the unique element of $\cP(\R)^{\otimes d}$ with $V\in L^1(\nu^*)$ whose density is strictly positive $\pi$-a.e.\ and solves the fixed point equation
 \begin{equation}\label{eq: fixed point}
     \nu_j^*(x_j) = \frac{e^{-\int_{\R^{d-1}} V(x) \nu^*_{-j}(\rd x_{-j}) }}{\int_\R e^{-\int_{\R^{d-1}} V(x) \nu^*_{-j}(\rd x_{-j}) } \nu^*_j(\rd x_j)} \quad \text{for $\mathcal{L}_1$-a.e.}\, x_j\in \R, \quad \text{for all } j\in[d]. 
 \end{equation}
 \end{theorem}

\section{Main Results} 
\label{sec:results}

We first state our results on the Lipschitz-continuity of the MFVI optimizer and the reward function (\cref{subsec:lip}), followed by a brief discussion of the proof methodology involving the lifted MFVI problem (\cref{se:methodology}). That discussion also introduces notions necessary for the second part (\cref{se: Differentiability}), where we state our results on differentiability of the MFVI optimizer.

\subsection{Lipschitz Stability}\label{subsec:lip}
\subsubsection{Lipschitz Stability of the Mean-Field Optimizer}\label{subsubsec:lip optimizer}

Our first result establishes global Lipschitz stability of the solution $\nu^*$ of \eqref{eq: MFVI} with respect to the target measure $\pi$.

\begin{theorem} \label{thm:Lip}  
Let $\pi\in\cP(\R^d)$ be  $\cC^2$ and $\alpha$-log-concave with potential function $V$, and let $\tilde\pi\in\cP(\R^d)$ be  $\cC^2$ and $\tilde\alpha$-log-concave with potential function $\tilde V$. Assume that\footnote{This condition implies~\eqref{eq:exponentialGrow2}, so that $\nu^*$ and $\tilde \nu^*$ are uniquely defined and admit strictly positive densities.} 
\begin{equation}
    \label{eq:exponentialGrow}
(\|\nabla {V(x)}\|^2+\|{\nabla \tilde{V}(x)}\|^2)e^{-\frac{\tilde\alpha}{2} {\|x\|^2}} \in L^1(\cL^d). %
\end{equation}
Let $\nu^*$ and $\tilde{\nu}^*$ be the solutions of \eqref{eq: MFVI} with target measures $\pi$ and $\tilde{\pi}$, respectively. Then 
\begin{equation}\label{eq:Lip}
    \cW_2(\tilde{\nu}^*,\nu^*)\leq \frac{\| \nabla\tilde V-\nabla V\|_{L^2(\tilde {\nu}^*)} }{\alpha}.
\end{equation}
\end{theorem}

Note that by exchanging the roles of $\tilde{\nu}^*$ and $\nu^*$, \cref{thm:Lip} implies the potentially tighter bound
$
\cW_2(\tilde{\nu}^*,\nu^*)\leq\min \{ \alpha^{-1}\| \nabla\tilde V-\nabla V\|_{L^2(\tilde {\nu}^*)} ,{\tilde\alpha}^{-1} \| \nabla\tilde V-\nabla V\|_{L^2(\nu^*)} \}
$ if~\eqref{eq:exponentialGrow} holds for both $\alpha$ and $\tilde\alpha$.
An analogous remark applies to many of the results below, and will not be repeated. 

\begin{remark}
With a view towards applications, observe that the right-hand side of our bound~\eqref{eq:Lip} can be approximated if samples can be drawn from the optimizer~$\tilde{\nu}^*$ (or $\nu^*$). \Cref{coro:Lip} below gives a bound for the right-hand side without requiring any access to the optimizer.
\end{remark}

\Cref{thm:Lip} can be interpreted as a mean-field analogue of the transport--information inequalities that have been studied in a substantial body of literature \citep{Djellout2004, Guillin2009TII1, Augeri2021, Lacker2023}. In this analogy, the quantity $\| \nabla \tilde V - \nabla V \|_{L^2(\tilde \nu^*)}$ in~\eqref{eq:Lip} replaces the relative Fisher information between $\tilde \pi$ and $\pi$, defined as $I(\tilde \pi \mid \pi) := \| \nabla \tilde V - \nabla V \|_{L^2(\tilde \pi)}^2$.  While the classic transport--information inequality $\cW_2(\tilde \pi, \pi) \leq \alpha^{-1} I(\tilde \pi \mid \pi)^{1/2}$ follows from the standard Talagrand transportation and log-Sobolev inequalities, \cref{thm:Lip} will be derived using optimal transport theory. We remark that the stability of the mean-field optimizer could also be analyzed by combining the transport--information inequality between $\tilde \nu^*$ and $\nu^*$ with a self-bounding argument. However, this approach leads to an undesirable dimension dependence, whereas our bound~\eqref{eq:Lip} is dimension-free.

\subsubsection{Lipschitz Stability of the Reward Function}\label{subsubsec:lip reward}
Using the definition of $\kl{\nu}{\pi}$ and the form of the density of $\pi$ with potential~$V$, the MFVI problem can be reformulated as maximizing the evidence lower bound (ELBO):
\begin{equation}\label{eq:ELBO}
  \sup_{\nu \in \cP(\R)^{\otimes d}} \ELBO (\nu,V), \quad \text{where} \quad \ELBO (\nu, V) :=  \int \left(-V(x) - \log \nu(x) \right)\nu(\rd x)
\end{equation}
if $\nu\ll \cL^d$ and $\ELBO (\nu, V) = -\infty$ if $\nu\not\ll \cL^d$. 
We define the \emph{reward function} of the mean-field problem as the maximum ELBO over the set of all product measures,
\begin{equation*}
\cR(V) := \sup_{\nu \in \cP(\R)^{\otimes d}} \ELBO (\nu, V).
\end{equation*}
The next result quantifies the stability of the maximum ELBO w.r.t.\ perturbations in $V$.

\begin{theorem}\label{thm:reward lip}
Let $\pi,\tilde\pi\in\cP(\R^d)$ be  $\cC^2$, $\alpha$-log-concave and $\beta$-log-smooth with potential functions $V,\tilde V$. Let $\nu^*, \tilde{\nu}^*$ be the solutions of \eqref{eq: MFVI} with target measures $\pi,\tilde{\pi}$. Then
\begin{equation*}
     |\cR(\tilde V) -  \cR(V) | \leq \frac{2\sqrt{\beta d}}{\alpha} \|\nabla \tilde V -\nabla V\|_{L^2(\tilde \nu^*)} + \frac{\beta}{2\alpha^2} \|\nabla \tilde V -\nabla V\|_{L^2(\tilde \nu^*)}^2 +  \|\tilde V - V\|_{L^2(\tilde \nu^*)}.
\end{equation*}
If $\tilde V,V$ are normalized such that $\int V \dd\tilde \nu^*=\int \tilde V \dd\tilde \nu^*$, then 
\begin{equation}\label{eq: reward lip 2nd}
|\cR(\tilde V) -  \cR(V)| \leq 
\frac{2\sqrt{\beta d} +1}{\alpha}\|\nabla \tilde V - \nabla V\|_{L^2(\tilde \nu^*)} +\frac{\beta}{2\alpha^2} \|\nabla \tilde V -\nabla V\|_{L^2(\tilde \nu^*)}^2.
\end{equation}
\end{theorem}

In contrast to \cref{thm:Lip}, the right-hand side in \cref{thm:reward lip} contains the dimension-dependent expression $\sqrt{d}\|\nabla \tilde V -\nabla V\|_{L^2(\tilde \nu^*)}$. Noting that $\|\nabla \tilde V -\nabla V\|_{L^2(\tilde \nu^*)}$ is generally expected to scale like $\sqrt{d}$, terms on the right-hand side of~\eqref{eq: reward lip 2nd} then scale like $d$. At the same time, both $\cR(V)$ and $\cR(\tilde V)$ also scale like $d$, suggesting that the factor $\sqrt{d}$ in~\eqref{eq: reward lip 2nd} is natural (and unavoidable). For example, when $\pi = \cN(0,\sigma^2 \mathrm{I}_d)$ and $\tilde \pi = \cN(0, \tilde \sigma^2 \mathrm{I}_d)$ for $\sigma,\tilde \sigma >0$, the left-hand side of~\eqref{eq: reward lip 2nd} becomes $\frac{d}{2} \left|\log \sigma^2 - \log \tilde \sigma^2 \right|$, while $\|\nabla \tilde V -\nabla V\|_{L^2(\tilde \nu^*)} =  \sqrt{d}  \tilde \sigma\left|1/\sigma^2 - 1/\tilde \sigma^2 \right|$.

\Cref{thm:reward lip} is helpful to establish the consistency of the learned prior in variational empirical Bayes procedures  (similarly as in~\citep[Theorem~3]{Mukherjee2023}). In particular, \cref{thm:reward lip} suggests that any log-concave prior lying within a ball of certain size around the ELBO-maximizing prior yields a consistent estimate of the true prior. In a different application, \cref{thm:reward lip} yields a stability result for distributed stochastic control with respect to the utility function; see \cref{sec: distributed sc}.

\subsubsection{Explicit Lipschitz Bounds}

We present an explicit version of the Lipschitz stability bound for the MFVI optimizer. In comparison with \cref{thm:Lip} which used moments under the MFVI optimizer in~\eqref{eq:Lip}, the next result replaces that possibly unknown quantity by an explicit bound. This is achieved by a novel density bound for the MFVI optimizer in terms of the log-concavity and log-smoothness constants, reported in \cref{lemma: minimizer bound}.

In applications, priors are often parametrized by a set $\Theta$, and one is interested in stability with respect to the parameter $\theta\in\Theta$. We state our result in that setting, taking $\Theta$ to be a subset of Euclidean space for concreteness.\footnote{The reader may notice, however, that \cref{coro:Lip} is ultimately a static result about a pair of fixed potentials,  regardless of parametrization. A formulation without parameter can be recovered by taking $\Theta=\{0,1\}$, so that $\|\tilde\theta-\theta\|=1$.}
\begin{corollary}\label{coro:Lip}
For $\theta\in\Theta$, let $\pi_\theta$ be  $\cC^2$, $\alpha_\theta$-log-concave, and $\beta_\theta$-log-smooth with potential function $V_\theta:\R^d\to \R$. Assume that there exists $L>0$ and $f:\R^d \to\R_+$ such that for all \(\theta, \tilde\theta \in \Theta \), 
\begin{equation*}
     \| \nabla V_{\tilde \theta}(x)-\nabla V_{\theta}(x)\| \leq L\|\tilde\theta-\theta\| f(x). 
\end{equation*}
Denote by  $\nu_\theta^*$ the MFVI optimizer for target $\pi_\theta$. Then
\begin{equation}\label{eq: curve lip general}
     \cW_2(\nu_{\tilde \theta}^*, \nu_\theta^*) \leq \frac{LC_{\theta, d }^{1/2}}{\alpha_{\tilde \theta}} \left(\int f(x)^2 \exp\left(-\frac{\alpha_{\theta}}{2}\|x -x^*_{\theta}\|^2\right)\dd x \right)^{1/2} \|\tilde\theta-\theta \|,
\end{equation}
where $x^*_{\theta} = \argmin_{x\in \R^d} V_{ \theta}(x)$ denotes the mode of $\pi_{\theta}$ and
\begin{equation*}
C_{\theta,d}:=   \left(\frac{2\pi}{\beta_{\theta}}\right)^{d/2} \exp\left(\frac{(\beta_{\theta}-\alpha_{\theta})d}{\alpha_{\theta}}  \left[2\kl{\cN(x^*_{\theta}, \alpha_{\theta}^{-1} \mathrm{I}_d)}{\pi_{\theta}}  + d \right] \right) < \infty.
\end{equation*}
Denoting by $S(d)$ the surface area of the unit sphere in $\R^d$, the following explicit bounds hold:
\begin{enumerate}[label=(\roman*)]
    \item If $f(x) = \|x\|^{p/2}$, then 
\begin{equation*}
     \cW_2(\nu_{\tilde \theta}^*, \nu_\theta^*) \leq \frac{LC_{\theta,d}^{1/2}}{\alpha_{\tilde \theta}}   \sqrt{ S(d)  2^{\frac{2p+d-4}{2}}\alpha_{ \theta}^{-\frac{d}{2}} \left[\|x^*_{\theta}\|^p \Gamma\left(\tfrac{d}{2}\right) + 2^{\frac{p}{2}}\,\alpha_{\theta}^{-\frac{p}{2}}\,\Gamma\!\left(\tfrac{p+d}{2}\right) \right]} \, \|\tilde \theta- \theta \|. 
\end{equation*}
\item If $f(x) = \exp \left(\frac{1}{2} \|x\| \right)$, then
\begin{equation*}
     \cW_2(\nu_{\tilde \theta}^*, \nu_\theta^*) \leq  \frac{LC_{\theta,d}^{1/2}}{\alpha_{\tilde \theta}}  \sqrt{S(d) e^{\|x^*_{\theta}\|} \int_0^\infty e^{r- \frac{\alpha_{\theta}}{2} r^2} r^{d-1}\dd r} \, \|\tilde \theta- \theta \|. 
\end{equation*}
\end{enumerate}
\end{corollary}

The key advantage of \cref{coro:Lip} is that the right-hand sides of the bounds can be obtained without computing an MFVI optimizer. Indeed, to evaluate the bounds, we first compute the mode $x^*_{\theta}$ using a first-order optimization method \citep{Beck2017} and then approximate the KL divergence $\kl{\cN(x^*_{\theta}, \alpha_{\theta}^{-1} \mathrm{I}_d)}{\pi_{\theta}}$ using an accept-reject algorithm \citep{Robert2004}, or further bound this quantity using the relative Fisher information via the log-Sobolev inequality \citep{Villani2003}. The remaining terms in the bounds can be computed using direct numerical integration.

\subsection{Methodology: The Lifted MFVI Problem}\label{se:methodology}

The purpose of this subsection is twofold: first, to introduce a lifting approach to the MFVI problem, which we consider the main methodological contribution. Second, to introduce notions needed to state the differentiability result in the subsequent \cref{se: Differentiability}. A key challenge in the analysis of the MFVI optimizer is that $\cP(\R)^{\otimes d}$ is not convex in the usual sense. \cite{Lacker2024} made use of the fact that $\cP_{2}(\R)^{\otimes d}$ is displacement convex to obtain the uniqueness in \cref{th:Lacker2024}. In the present work, a key innovation is to use \emph{linearized} optimal transport: Fix the reference measure $\rho=\cN(0,\mathrm{I}_d)$. Since for each measure $\mu \in \cP_2(\R^d)$ there is a unique optimal transport map $T_{\rho \to \mu}$, we can parametrize such measures $\mu$ by their optimal transport maps, and this turns out to be particularly useful for the set $\cP_2(\R)^{\otimes d}$. In fact, for $\mu=\bigotimes_{i=1}^d\mu_i\in\cP(\R)^{\otimes d}$, the components $T_i$ of $T_{\rho \to \mu}=(T_1,\dots,T_d)$ are the optimal transport maps from $\rho_1$ to the marginals $\mu_i$. Each such transport problem being scalar, the map $T_i$ has a simple form, namely, $T_i = Q_i \circ \Phi$ where $\Phi$ is the cumulative distribution function of $\cN(0,1)$ and $Q_i$ is the quantile function of $\mu_i$.

Recall that the space $(\cP_2(\R^d), \cW_2)$ can be endowed with a formal Riemannian structure~\citep{otto2001geometry}. The tangent space of the submanifold $\cP_2(\R)^{\otimes d} \subset \cP_2(\R^d)$ at the base measure $\rho$ is identified (see \cite[Section~3]{lacker2023independent} for details) as
\begin{equation}
    {\rm Tan}_{\rho}\cP_{2}(\R)^{\otimes d} = \bigoplus_{i=1}^d {\rm Tan}_{\rho_i}\cP_{2}(\R) = (L^2(\rho_1))^d,
\end{equation}
where ${\rm Tan}_{\rho_1}\cP_{2}(\R) = \overline{ \left\{\varphi': \varphi \in \cC_c^\infty(\R)\right\}}^{L^2(\rho_1)} = L^2(\rho_1)$.
Once we represent each $\mu \in \cP_2(\R)^{\otimes d}$ via the optimal transport map $T_{\rho \to \mu} \in {\rm Tan}_{\rho}\cP_{2}(\R)^{\otimes d}$, \eqref{eq: MFVI} becomes a convex optimization problem over optimal transport maps, or more explicitly, the set of maps $T(u)=(T_1(u_1),\ldots, T_d(u_d))$ such that each $T_i$ is the gradient of a convex function on~$\R$. Specifically, since we know that the MFVI optimizer is absolutely continuous and the map $T_{\rho \to \mu}$ is the gradient of a strictly convex function on $\R^d$ when $\mu$ is absolutely continuous, \eqref{eq: MFVI} becomes a minimization over the convex cone
\begin{equation}\label{eq:convex cone}
    \cH_+ := \left\{(T_1, \dots, T_d) \in \cH(\rho) : T_i' > 0 \text{ a.e.} \right\} \subset \cH(\rho),
\end{equation}
where $\cH(\rho) := (\cH^1(\rho_1))^d \subset {\rm Tan}_{\rho}\cP_{2}(\R)^{\otimes d}$ and $\cH^1(\rho_1)$ is the Gaussian Sobolev space over~$\R$ (see \cref{se: notation and prelims}).
We show in \cref{lemma:lifted KL} that~\eqref{eq: MFVI} is equivalent to 
\begin{equation}\label{eq: lifted MF}
     \min_{ (T_1,\ldots, T_d)\in \cH_+ }  \cF_V(T_1,\ldots, T_d),  \tag{L-MFVI}
\end{equation}
where $\cF_V$ is the convex functional on $\cH(\rho)$ defined as
\begin{equation}\label{eq: F_V}
    \cF_V( T_1,\ldots, T_d): =  
        - \sum_{i=1}^d \int_\R \log (T_i'(u_i)) \rho_1(\rd u_i) +\int V( T_1(u_1),\ldots, T_d(u_d)) \rho(\rd u)
\end{equation}
if $ (T_1,\ldots, T_d) \in \cH_+$, and $\infty$ otherwise. We call~\eqref{eq: lifted MF} the \textit{lifted problem} of~\eqref{eq: MFVI}. %
\Cref{thm:Lip} will be derived by analyzing the first-order optimality condition of the convex problem~\eqref{eq: lifted MF}.

\subsection{Differentiability of the MFVI Optimizer}\label{se: Differentiability}

Our second main result studies the derivative of the mean-field approximation with respect to a  perturbation of the target potential.  More precisely, we consider potentials $\{V_\theta\}_{\theta\in \Theta}$ depending suitably on a parameter $\theta$ and show that the optimizer $T^\theta$ of \eqref{eq: lifted MF} for the target measure $\pi_\theta \propto e^{-V_\theta}$ is differentiable in $\theta$. Moreover, we characterize the derivative as the solution to a partial differential equation. We assume that $\Theta$ is an interval, but note (see \cref{rk: Differentiable}) that this includes the case of a parameter varying along a segment in a high-dimensional space.
\begin{theorem}\label{th: Differentiable}
 Let $0<\alpha\leq \beta$. Let $\{V_\theta\}_{\theta\in \Theta} \subset \cC^2(\R^d)$ be a family of $\alpha$-strongly 
 convex and $\beta$-smooth functions indexed by an open interval $\Theta$. Let $\theta_0\in\Theta$ and assume that 
\begin{enumerate}[label=(\roman*)]
    \item $V_\theta \to V_{\theta_0}$ pointwise as $\theta \to \theta_0$;
    \item for some $\vae>0$, the mapping
    \begin{equation*}
        \Theta \ni \theta \mapsto \nabla V_\theta \in L^2(\cN(0, (1+\vae)\alpha^{-1}\mathrm{I}_d))
    \end{equation*}
    is well-defined and Fréchet differentiable at $\theta_0 \in \Theta$ with derivative $\partial_\theta \nabla V_{\theta_0}: \R^d \to \R^d$ such that $x\mapsto \partial_\theta \nabla V_{\theta_0}(x)$ is Lipschitz.
    \item we have
    \begin{equation}\label{assum:differentiable}
        \| \nabla V_{\theta} - \nabla V_{\theta_0} \|_{L^p(\cN(0, (1+\vae)\alpha^{-1}\mathrm{I}_d))} \leq C_{\theta_0} |\theta - \theta_0| \quad \text{for some } p > 2.
    \end{equation}
\end{enumerate}
  Then $ \Theta \ni \theta \mapsto T^\theta\in \cH(\rho)$ 
  is Fréchet differentiable at $\theta_0\in \Theta$. The derivative $\partial_\theta T^{\theta_0}\in \cH(\rho)$ can be characterized as the unique solution $S\in\cH(\rho)$ of the equation
\begin{equation} \label{eq:derivative-formula}
        \cB_{\theta_0}(S, R) = - \int \inner{ \partial_\theta \nabla V_{\theta_0} \circ T^{\theta_0}(u)}{R(u)} \rho(\rd u) \quad \text{for every}\,\, R\in \cH(\rho),
\end{equation}
where $\cB_{\theta_0}: \cH(\rho) \times \cH(\rho) \to \R$ is the bilinear operator given by  
\begin{equation}\label{eq:bilinear op}
   \cB_{\theta_0}(S,R) = \int \langle [\nabla^2 V_{\theta_0}(T^{\theta_0}(u))] S(u),  R(u) \rangle \rho(\rd u)+ \int \sum_{i=1}^d\frac{S_i'(u_i)R_i'(u_i)}{((T^{\theta_0}_i)'(u_i))^2} \rho(\rd u).
\end{equation}
\end{theorem}

\Cref{rk: derivative formula explained} below elaborates on the meaning of the characterization~\eqref{eq:derivative-formula} and the bilinear operator~\eqref{eq:bilinear op}. Before that, some comments on the conditions of \cref{th: Differentiable} are in order.

\begin{remark}\label{rk: Differentiable}
\begin{enumerate}[label = (\alph*)]
\item The assumptions in \cref{th: Differentiable} can be understood as follows. On the one hand, the Fréchet differentiability of $\Theta \ni \theta \mapsto \nabla V_\theta \in L^2(\cN(0, (1+\vae)\alpha^{-1}\mathrm{I}_d))$ implies that $ \partial_\theta \nabla V_{\theta_0} \in L^2(\nu_{\theta_0}^*)$; see \cref{lemma: L^2}. This ensures that the right-hand side of~\eqref{eq:derivative-formula} is finite for all $R \in \cH(\rho)$. On the other hand, the $\alpha$-convexity and $\beta$-smoothness condition of $\{V_\theta\}_{\theta \in \Theta}$  ensures that the bilinear operator $\cB_{\theta_0}$ is well-behaved, so that the Lax--Milgram theorem allows us to infer the existence and uniqueness of the derivative~$\partial_\theta T^{\theta_0}$.  

\item  The Lipschitz condition~\eqref{assum:differentiable} can be replaced by  $\| \nabla  V_{\theta}- \nabla V_{\theta_0} \|_{L^p(\nu_{\theta_0}^*)}\leq  C_{\theta_0} |\theta-\theta_0|$. However, \eqref{assum:differentiable} has the advantage of being verifiable without access to the MFVI optimizer $\nu_{\theta_0}^*$. Indeed, given query access to $\nabla V_\theta$, one can sample from $\cN(0, (1 + \varepsilon)\alpha^{-1} \mathrm{I}_d)$ and verify assumption~\eqref{assum:differentiable} for any pair of parameters  $\theta, \theta_0$. 

\item While we assumed that $\Theta$ is an interval, our argument is applicable to multivariate (or even infinite-dimensional) parameters. E.g., for $\theta, \theta_0 \in \Theta \subset \R^k$, we can consider the line segment $\{(1 - t)\theta + t \theta_0 : t \in (-\delta, 1 + \delta)\}$ for some $\delta > 0$ and accordingly define an auxiliary parameter space $\tilde \Theta := (-\delta, 1 + \delta)$. See \cref{example: contaminated model} for an illustration based on a Huber contamination model in robust statistics.

\item  The differentiability of $\theta \mapsto T^\theta$ also holds in $L^\infty_{\rm loc}(\R^d)$. Indeed, as $\cH(\rho)$ is contained in the local Sobolev space $\cH^1_{\rm loc}(\R^d)$ \cite[Proposition~1.5.2]{bogachev1998gaussian}, Morrey's inequality \cite[Theorem~9.1.2]{brezis2011functional} yields that $\cH^1_{\rm loc}(\R^d)$ is embedded in $L^\infty_{\rm loc}(\R^d)$ and the map $\theta \mapsto T^\theta \in L^\infty_{\rm loc}(\R^d)$ is Fr\'echet differentiable at any $\theta_0 \in \Theta$.
\end{enumerate}
\end{remark}

The characterization of the Fréchet derivative $\partial_\theta T^{\theta_0}$ in \eqref{eq:derivative-formula} is related to \citep[Theorem~1.2]{gonzalez2024linearization} which studies the differentiability of the optimal transport map with respect to a parameter when the reference and target measures are both supported on bounded, convex sets and admit smooth densities. In that setting, the derivative is characterized as the classical solution to a second-order elliptic partial differential equation, in contrast to the weak formulation provided in \cref{th: Differentiable}.

Solving for $\partial_\theta T^{\theta_0}$ yields a first-order approximation to the MFVI optimizer $\nu^*_\theta$. Define $\tilde \nu_\theta := \left( T^{\theta_0} + (\theta - \theta_0)\, \partial_\theta T^{\theta_0} \right)_\# \rho$. Then, \cref{th: Differentiable} implies that
$\cW_2(\tilde \nu_\theta, \nu^*_\theta) = o(|\theta - \theta_0|)$, allowing one to compute the MFVI optimizer only at $\theta_0$ and adjust for small perturbations in $\theta$ via a first-order correction. This has potential applications, e.g, in transfer learning. Finally, we remark that under sufficient regularity, the classical formulation in \eqref{eq: classical formal} enables computation of $\partial_\theta T^{\theta_0}$ by solving a system of ordinary differential equations.

\vspace{1em}

As mentioned in \cref{se:methodology}, the component $T_i^\theta$ of $T^\theta$ is the optimal transport map from $\rho_i$ to $\nu_{\theta, i}^*$ and has the form $T_i^\theta = Q_i^\theta \circ \Phi$ where $\Phi$ is the cdf of $\cN(0,1)$ and $Q_i^\theta$ is the quantile function of the $i$-th marginal $\nu_{\theta,i}^*$; that is, $Q_i^\theta$ is the optimal transport map from ${\rm Unif}[0,1]$ to~$\nu_{\theta,i}^*$. Writing $Q^\theta = (Q_1^\theta,\dots, Q_d^\theta)$, this leads to $\|T^\theta - T^{\theta_0}\|_{L^2(\rho)} = \|Q^\theta - Q^{\theta_0}\|_{(L^2([0,1]))^d}$. As a consequence, \Cref{th: Differentiable} yields the following result.

\begin{corollary}
  In the setting of \cref{th: Differentiable}, let $Q_i^\theta$ denote the quantile function of the $i$-th marginal of $\nu_{\theta}^*$. Then the map $\theta \mapsto Q^\theta = (Q_1^\theta,\ldots, Q_d^\theta) \in (L^2([0,1]))^d$ is Fr\'echet differentiable. 
\end{corollary}

This result demonstrates that the quantile function of the MFVI optimizer is stable with respect to perturbations of the target potential, enabling its use in rank-based inference and robust statistical estimation~\citep{HallinBarrio2021, ghosal2022multivariate}.

\begin{remark}\label{rk: derivative formula explained}
    We can understand the characterization~\eqref{eq:derivative-formula} as the weak solution concept for a more straightforward (but technically inconvenient) variational condition, \eqref{eq: classical formal} below. Indeed, the bilinear operator $\cB_{\theta_0}$ in~\eqref{eq:bilinear op} arises naturally as the second-order derivative of the lifted MFVI problem. To see this, suppose that $T^{\theta_0}$ is smooth and each $(T^{\theta_0}_i)'$ is bounded away from zero and infinity. Define the operator $\Gamma_i: \Theta \times \cH(\rho) \mapsto \cH(\rho)$ by 
    \begin{equation*}
        \Gamma_i(\theta, T)(u_i) =   - \frac{(T_i)''(u_i)}{\left((T_i)'(u_i)\right)^2} - \frac{u_i}{(T_i)'(u_i)} + \int_{\R^{d-1}} \partial_i V_{\theta}\left(T(u)\right) \, \rho_{-i}(\rd u_{-i}) ,\quad i \in [d].
    \end{equation*}
    The first-order optimality condition of \eqref{eq: lifted MF} entails that 
    \begin{equation}\label{eq:strong-formulation}
       \Gamma_i(\theta_0,T^{\theta_0})= 0,\quad \forall \,i\in [d].
    \end{equation}  
     Taking the directional derivative of $\Gamma_i$ at $T^{\theta_0}$ along a smooth direction $S =(S_1,\ldots, S_d) \in  (\cC^\infty_c(\R))^d$ and using integration by parts, we deduce that   
    \begin{equation}\label{eq: classical is weak}
    \sum_{i=1}^d\inner{D_{T} \Gamma_i(\theta_0, T^{\theta_0})(S)}{R_i}_{L^2(\rho_1)} = \cB_{\theta_0}(S,R) \quad\mbox{for all } S, R \in (\cC^\infty_c(\R))^d. 
    \end{equation}
    Comparing~\eqref{eq: classical is weak} to the characterization~\eqref{eq:derivative-formula} in \cref{th: Differentiable}, we observe that the Fr\'echet derivative $\partial_\theta T^{\theta_0}$ is formally the weak solution to the following system of equations,
    \begin{equation}\label{eq: classical formal}
    D_{T} \Gamma_i(\theta, T^{\theta_0})(S)(u_i) = -\int_{\R^{d-1}}  \partial_\theta \nabla V_{\theta_0} \circ T^{\theta_0} (u_i,u_{-i})\rho_{-i}(\rd u_{-i}),\quad \forall\, u_i\in \R, i\in [d].
    \end{equation}

    While this formal derivation motivates our result, our proof takes a different route. Indeed, working with~\eqref{eq:strong-formulation} would require pointwise control of the first and second derivatives of $T^{\theta_0}$, which seems challenging to obtain. As highlighted by the relation~\eqref{eq: classical is weak}, we may instead work with the bilinear operator $\cB_{\theta_0}$ which is better-behaved in our setting (see the proof of \cref{th: Differentiable}). %
    The main technical obstacle in our approach is that functions in $\cH(\rho)$ are only defined $\rho$-a.e.\ and hence the implicit function theorem cannot be applied. The main step to overcome this is a novel $L^p$ estimate on the difference $T^\theta - T^{\theta_0}$; see \cref{pr: Lp estimate}. Further comments on the proof strategy are deferred to \cref{sec:proof-differentiable} as they require introducing additional objects.
\end{remark}

Remarkably, the following $\cW_p$ stability bound for $p > 2$ is obtained as a byproduct of our proof strategy (it is part of \cref{pr: Lp estimate} which additionally gives a bound for the derivatives of the \eqref{eq: lifted MF} optimizers). Comparing with the dimension-free bound in \cref{thm:Lip} for $p=2$, the price we pay for increasing~$p$ is that the bound depends explicitly on the dimension~$d$. 

\begin{proposition}\label{prop: Wp Lipschitz}
 In the setting of \cref{th: Differentiable}, for any $p\geq 2$, 
    \[
    \cW_p(\nu_\theta, \nu_{\theta_0}) \leq d^{\frac{p-2}{2}} \frac{\alpha + \sqrt{d}\beta}{\alpha^2}\|\nabla V_{\theta}(T^{\theta}) - \nabla V_{\theta_0}(T^{\theta})\|_{L^p(\rho)}.
    \]
\end{proposition}

\section{Applications} \label{sec:applications}

In this section, we present several applications of our main results. In \cref{sec:variational bayes}, we quantify the sensitivity of variational Bayes in the context of Bayesian linear regression, prior swapping, and when priors belong to a $\varepsilon$-contamination class. In \cref{sec: quantitative BvM}, we establish a quantitative Bernstein--von Mises theorem for MFVI which states that the variational posterior is well-approximated by a Gaussian when the sample size is large. Finally, in \cref{sec: distributed sc}, we provide an application to distributed stochastic control, quantifying the stability of the optimal value with respect to the utility function.

\subsection{Variational Bayesian Robustness}\label{sec:variational bayes}
Our first application is in the context of Bayesian statistics. Suppose we observe a collection of data points $\bx_n = \{x_1, \dots, x_n\} \in \R^n$ generated from a model with Lebesgue density $\rmp(\bx_n \mid \bz_d)$, conditional on latent variables $\bz_d = \{z_1, \dots, z_d\} \in \R^d$. We place a prior distribution $\rmp(\bz_d) \in \cP(\R^d)$ on the latent variables. Together, the prior and likelihood define a joint distribution over $(\bx_n, \bz_d)$ with potential $V(\bx_n, \bz_d) = - \log \rmp(\bz_d) - \log \rmp(\bx_n \mid \bz_d)$. The central object in Bayesian inference is the posterior distribution, given by the conditional law of $\bz_d$ given $\bx_n$: 
\begin{equation*}
    \pi \left(\rd \bz_d \mid \bx_n \right) = \frac{\exp \left(- V\left( \bx_n, \bz_d \right) \right)}{\int_{\R^d} \exp \left(- V\left( \bx_n, \bz_d \right) \right) \dd \bz_d}. 
\end{equation*}
When $\bz_d$ is high-dimensional, posterior sampling methods such as Markov chain Monte Carlo (MCMC) can be computationally prohibitive. An empirically successful alternative is to approximate the posterior distribution \(\pi(\cdot \mid \bx_n)\) using MFVI \citep{Wainwright2008,Blei2017}.

Many probabilistic models have the form $V_\theta(\bx_n, \bz_d)$, where $ \theta \in \Theta$ parametrizes either the prior or the likelihood. This setting induces a family of posterior distributions $ \{\pi_\theta\}_{\theta \in \Theta}$, with the corresponding MFVIs $ \{\nu_\theta^*\}_{\theta \in \Theta}$.

We analyze the sensitivity in Wasserstein distance of $\nu_\theta^*$ to perturbations of $\theta$ around some reference point $\theta_0 \in \Theta$. By the Kantorovich--Rubinstein duality \cite[Theorem~1.14]{Villani2003}, a bound in Wasserstein distance controls the fluctuation of the posterior means of Lipschitz statistics, including linear statistics $\varphi(x) = \langle \gamma, x \rangle$ for some vector $\gamma \in \R^d$ or the norm $\varphi(x) = \|x\|_1$ (see \cref{cor:prior-swapping} for an example).

\subsubsection{Variational Empirical Bayes} \label{sec:VEB}
Empirical Bayes \citep{Berger1986} replaces an unknown parameter $\theta_0$ in a probabilistic model by a frequentist estimate $\hat{\theta}$. In this context, our stability bound quantifies how the estimation uncertainty of $\hat{\theta}$ propagates to the uncertainty of statistical inference based on $\nu_{\hat{\theta}}^*$. We illustrate this in the examples of a linear regression model.

\begin{example}[Bayesian Linear Regression]\label{example:linear-model} Suppose we observe data points $\{(y_i, X_i) \}_{i = 1}^n$, where $y_i \in \R$ and $X_i \in \R^d$. Let $y = (y_1, \dots, y_n) \in \R^n$ and $\bX^\top = (X_1, \dots, X_n) \in \R^{d \times n}$. Consider the model 
\begin{equation*}
    y = \bX \beta  + \epsilon, \quad \epsilon \sim \cN \left(0, \tau^{-1} \mathrm{I}_n \right). 
\end{equation*}
Here $\beta \in \R^d$ are the coefficients of interest. Assume that $\beta_i \iid \mu$ where $\mu(\rd \beta_1) \propto \exp \left(- v(\beta_1)\right)\dd \beta_1$ for some $v: \R \mapsto \R$. Set $\bA := \bX^\top \bX, \bw := \bX^\top y$. The posterior density of $\beta$ given $y$ and $\bX$  has the form $\pi_\tau \left(\rd \beta\right) \propto \exp \left( - V_\tau(\beta)\right)\dd\beta$, where 
\begin{equation*}
    V_\tau(\beta) = \sum_{i = 1}^d v(\beta_i) + \frac{\tau}{2} \beta^T \bA \beta - \tau \bw^T \beta. 
\end{equation*}
MFVI is frequently used to approximate the posterior of high-dimensional linear models, both in statistical methodology \citep{Carbonetto2012, Wang2020EBVI-VS} and in statistical theory \citep{Mukherjee2021, Mukherjee2023,Lacker2024}. In practice, however, the unknown precision parameter $\tau$ is typically replaced by an estimate $\hat \tau$. The following result shows that the MFVI optimizer is Lipschitz stable in $\cW_2$ with respect to the estimation error in $\hat \tau$.
\begin{corollary}\label{co:linear-model}
Assume that $v$ is $\alpha_0$-convex for some $\alpha_0 \geq 0$ and that $\bA \succeq \alpha_1 \mathrm{I}_d$ for some $\alpha_1 > 0$. For $\tau > 0$, let $\nu_{\tau}^*$ be the MFVI optimizer for target $\pi_{\tau}$. Then
\begin{equation} \label{lm-bound}
         \cW_2(\nu_{\hat \tau}^*,\nu_\tau^*)\leq \frac{\| \bA \beta  - \bw \|_{L^2(\nu^*_{\hat \tau})}}{\alpha_1 \tau + \alpha_0}|\hat \tau - \tau| 
\end{equation}
and the full approximation error satisfies
\begin{equation*}
 \cW_2(\nu_{\hat \tau}^*,\pi_\tau)\leq \frac{\| \bA \beta  - \bw \|_{L^2(\nu^*_{\hat \tau})}}{\alpha_1 \tau  + \alpha_0}|\hat \tau - \tau| + \frac{\sqrt{2 \sum_{1 \leq i < j \leq d} |\bA_{ij}|^2}}{\left(\alpha_1 \tau + \alpha_0 \right)^{3/2}}.     
\end{equation*}
\end{corollary}
\begin{proof}
In view of $\nabla^2 V_{\hat \tau} \succeq \alpha_1 \hat \tau + \alpha_0$ and $\nabla^2 V_\tau \succeq \alpha_1 \tau + \alpha_0$, the posteriors $\pi_{\hat \tau}$ and $\pi_\tau$ are log-concave, with log-concavity constants $\alpha_1 \hat \tau + \alpha_0$ and $\alpha_1 \tau + \alpha_0$, respectively. \Cref{thm:Lip} thus yields
\begin{equation*}
        \cW_2(\nu_{\hat \tau}^*,\nu_\tau^*)\leq  \frac{\|\nabla V_{\hat \tau} -\nabla V_\tau   \|_{L^2(\nu_{\hat \tau})}}{\alpha_1 \tau  + \alpha_0} =\frac{\| \bA \beta  - \bw \|_{L^2(\nu_{\hat \tau})}}{\alpha_1 \tau  + \alpha_0}|\hat \tau - \tau|.
\end{equation*}
The bound on full approximation error follows by the triangle inequality $\cW_2(\nu_{\hat \tau}^*,\pi_\tau)\leq\cW_2(\nu_{\hat \tau}^*,\nu_\tau^*)+\cW_2(\nu_{\tau}^*,\pi_\tau)$ and the bound for $\cW_2(\nu_{\tau}^*,\pi_\tau)$ given in \cite[Theorem~1.1]{Lacker2024}.
\end{proof}
\Cref{co:linear-model} guarantees that $\nu_{\hat \tau}^*$ is close to $\nu_{\tau}^*$ when $\hat \tau$ is close to $\tau$.  The Lipschitz constant in \eqref{lm-bound} may be interpreted as the posterior predictive mean squared error (MSE) of $\bw$ under the variational posterior $\nu_{\hat \tau}$. If $\hat{\tau}$ is a consistent estimate of $\tau$, one can empirically quantify the local sensitivity of $\nu_\tau^*$ at $\tau = \hat \tau$ by drawing Monte Carlo samples $ \beta^{(1)}, \dots, \beta^{(m)}$ from $\nu_{\hat \tau}$ and evaluating the bound 
$\frac{\sqrt{\sum_{j = 1}^m \|\bA \beta^{(j)} - \bw\|^2}}{\sqrt{m}(\alpha_1 \tau + \alpha_0)}$.
A large bound suggests that $\hat \nu_\tau$ is highly sensitive to the estimation error in $\hat \tau$, thus more data should be used to obtain a more accurate estimate of  $\hat \tau$ for reliable downstream inference. We remark that~\cref{co:linear-model} remains valid even when $\alpha_0$ is negative, provided that $\alpha_1 \tau  + \alpha_0 > 0$.
\end{example}

\subsubsection{Robustness under Perturbation in Prior} 
Our second application concerns the classical notion of Bayesian robustness in statistics \citep{Berger1985book} and econometrics \citep{Giacomini2021surveyRobustBayes}. Two main approaches thereof are \textit{global robustness}, which examines the range of posterior means as the prior varies over a family of distributions \citep{Berger1985book}, and \textit{local robustness}, which quantifies the sensitivity of posterior functionals to infinitesimal changes in the hyperparameter \citep{Gustafson2000}. Within the local robustness framework, \cite{Gustafson1996} investigates the sensitivity of the posterior mean, while \cite{Gustafson2002} quantifies the sensitivity of the full posterior distribution under perturbations of the prior, using the total variation metric. \cite{Giordano2018} study local prior robustness in the context of MFVI, deriving sensitivity measures that quantify the effect of model perturbations on the variational posterior mean. In our setting, a Lipschitz bound of the form \cref{coro:Lip}, namely $\cW_2(\nu_{\tilde \theta}^*, \nu_\theta^*) \leq L \|\tilde \theta - \theta\|$, implies a uniform bound of $\ell L$ on the local sensitivity of $\ell$-Lipschitz posterior means with respect to $\theta$.
\begin{example}[Variational Bayes with Prior Swapping]
    In the prior swapping problem \citep{ Neiswanger2021}, the variational posterior $\tilde{\nu}^*$ is computed under a surrogate prior $\tilde{\rmp}$, while the desirable target is the variational posterior $\nu^*$ associated with the prior $\rmp$. The surrogate prior $\tilde{\rmp}$ is often chosen for computational tractability---for instance, $\tilde{\rmp}$ might be conjugate to the likelihood, yielding a posterior with known parametric form \citep{Diaconis2007}, or one that provides closed-form updates in coordinate ascent variational inference (CAVI) \citep{Wang2013}.

Existing approaches to  prior swapping typically require estimating the likelihood ratio between $\tilde{\rmp}$ and $\rmp$, a procedure that  becomes computationally expensive in high dimensions when the normalizing constant of $\rmp$ is unknown, e.g., when the prior is specified via an energy-based model \citep{LeCun2006}. Using the stability bound of \cref{thm:Lip}, we propose an alternative approach to prior swapping that avoids directly computing the density ratio.
\begin{corollary} \label{cor:prior-swapping}
Assume that $\log \rmp(\bx_n \mid \cdot)$ is $\alpha_{n, d}$-concave for some $\alpha_{n, d} \in \R$, and that the priors $\rmp$ and $\tilde{\rmp}$ are $\tilde{\alpha}_d$-log-concave for some $\tilde{\alpha}_d\in\R$ such that $\alpha_{n,d}+\tilde{\alpha}_d>0$. Let $\nu^*$ and $\tilde{\nu}^*$ denote the MFVI optimizers for the posteriors under  likelihood $\rmp(\bx_n \mid \cdot)$ and priors $\rmp$ and $\tilde{\rmp}$, respectively. Then 
\begin{equation} \label{eq:bound-prior-swapping}
\cW_2(\tilde{\nu}^*, \nu^*) \leq \frac{\| \nabla \log \tilde{\rmp} - \nabla \log \rmp \|_{L^2(\tilde{\nu}^*)}}{\alpha_{n, d} + \tilde{\alpha}_d}.
\end{equation}
In particular, any $\ell$-Lipschitz statistic $\varphi: \R^d \to \R$ satisfies
\begin{equation} \label{eq:interval-prior-swapping}
\left| \EE{\nu^*}{\varphi(X)} - \EE{\tilde \nu^*}{\varphi(X)}  \right|  \leq\frac{\ell \| \nabla \log \tilde{\rmp} - \nabla \log \rmp \|_{L^2(\tilde{\nu}^*)}}{\alpha_{n, d} + \tilde{\alpha}_d}.
\end{equation}
\end{corollary}
\begin{proof}
The bound~\eqref{eq:bound-prior-swapping} is a direct consequence of \cref{thm:Lip} while \eqref{eq:interval-prior-swapping} follows from~\eqref{eq:bound-prior-swapping} via the Kantorovich--Rubinstein formula \cite[Equation~(5.11)]{villani2009optimal} and $\cW_1(\tilde{\nu}^*, \nu^*)\leq \cW_2(\tilde{\nu}^*, \nu^*)$. 
\end{proof}

It is often easy to evaluate the gradient of the log-prior even when it is not easy to compute the normalizing constants. Given a statistic $\varphi:\R^d \to \R$ of interest, the bound~\eqref{eq:interval-prior-swapping} provides an \textit{uncertainty region} for $\EE{\nu^*}{\varphi(X)}$ given samples from the surrogate posterior $\tilde{\nu}^*$:
\begin{equation*}
     \EE{\nu^*}{\varphi(X)} \in \big[\EE{\tilde \nu^*}{\varphi(X)}  - \delta, \EE{\tilde \nu^*}{\varphi(X)}  +\delta \big], \quad \delta:= \frac{\beta \| \nabla \log \tilde{\rmp} - \nabla \log \rmp \|_{L^2(\tilde{\nu}^*)}}{\alpha_{n, d} + \tilde{\alpha}_d}. 
\end{equation*}
\end{example}

\begin{example}[Robustness for $\epsilon$-contamination class]\label{example: contaminated model}
To assess the influence of priors on posterior inference, the global robustness approach considers the range of posteriors over a family of priors. Here we focus on a particularly appealing class, the $\epsilon$-contamination class, commonly used in robust Bayesian analysis \citep{Berger1985book,Berger1986,Sivaganesan1989}. Given a reference prior $\rmp$ and a suitable perturbing distribution $\rmq$, we consider the linear interpolation
\begin{equation*}
    \rmp_\epsilon = (1 - \epsilon) \rmp + \epsilon \rmq. 
\end{equation*}
Assume that $\log \rmp(\bx_n \mid \cdot)$ is $\alpha_{n, d}$-concave for some $\alpha_{n, d} > 0$, and that $\rmp_\epsilon$ is $\alpha_\epsilon$-log-concave for some $\alpha_\epsilon\in\R$ with $\alpha_\epsilon + \alpha_{n, d} > 0$, with $\alpha_\epsilon$ continuous in $\epsilon$. Let $\nu^*$ and $\nu_\epsilon^*$ denote the mean-field approximations to the posteriors under likelihood $\rmp(\bx_n \mid \cdot)$ and priors $\rmp$ and $\rmp_\epsilon$, respectively. \cref{thm:Lip} implies the stability bound
\begin{equation*}
\cW_2(\nu_\epsilon^*,\nu^*) \leq \frac{\| \nabla \log \rmp_\epsilon -\nabla \log \rmp \|_{L^2(\tilde \nu^*)} }{ \alpha_{n, d} + \alpha_\epsilon} = \left\| \frac{\rmq}{ \rmp_\epsilon} \left(\nabla \log \rmq - \nabla \log \rmp \right)\right\|_{L^2(\nu^*)} \frac{\epsilon}{\alpha_{n, d} + \alpha_\epsilon}. 
\end{equation*}
As $\epsilon \to 0$, the local sensitivity $\left.\frac{\dd}{\dd \epsilon} \cW_2(\nu_\epsilon^*,\nu^*)  \right|_{\epsilon = 0}$ is upper bounded by 
\begin{equation*}
    (\alpha_0 + \alpha_{n, d})^{-1} \left\| \frac{\rmq}{\rmp}(\nabla \log \rmq - \nabla \log \rmp)\right\|_{L^2(\nu^*)} .
\end{equation*}
\end{example}

\subsection{Quantitative Bernstein--von Mises Theorem for Mean-Field Variational Posterior}\label{sec: quantitative BvM}
The well-known Bernstein--von Mises (BvM) theorem characterizes the asymptotic stability of the posterior for finite-dimensional parameter spaces (\cite[Section~10]{VdV2000}; see also \cite{Miller2021} and the references therein). It tells us that, when the sample size is large, the posterior distribution is approximately a Gaussian centered at the true parameter. Recent works have extended BvM to semiparametric models \citep{Bickel2012}, variational inference \citep{Wang2019}, misspecified models \citep{Miller2021} and models with growing dimensions \citep{Katsevich2023}.

The main result of this subsection is a quantitative BvM theorem for MFVI. We state two versions; version (i) assumes smoothness on $\R^d$ whereas (ii) imposes a smoothness condition only on a ball, and a growth condition outside the ball. 
  
\begin{theorem}\label{thm:non-asymptotic normal}
Let $\pi_n\in\cP(\R^d)$ with $\pi_n(\rd x) \propto \exp \left(- n f_n(x) \right)\dd x$ where $f_n$ is $\cC^2$ and $\alpha_n$-strongly convex. Let $x_n^*=\argmin f_n$ and $\gamma_n^* := \cN \left(x_n^*, D_n^{-1} \right)$, where $D_n$ is the diagonal part\footnote{By the \emph{diagonal part} of a matrix $A\in\R^{n\times n}$, we mean the diagonal matrix $B\in\R^{n\times n}$ with the same diagonal as $A$.} of the matrix $n \nabla^2 f_n(x_n^*)$. Let $\nu_n^*$ be the MFVI optimizer for target $\pi_n$. 

\begin{enumerate}[label=(\roman*)]
\item If $f_n$ is $b_n$-smooth, the MFVI optimizer satisfies
\begin{equation} \label{AN-W2-bound1}
    \cW_2^2 \left( \nu_n^*, \gamma_n^* \right) \leq \frac{4 b_n^2 d}{\alpha_n^3 n}. 
\end{equation}
\item If $\nabla^2 f(x_n^*)\preceq b_n\mathrm{I}_d$ and there exist $\tau_n \in [0, n \alpha_n)$ and $s_n > \sqrt{(d + 2)/(n \alpha_n - \tau_n)}$ such that $\nabla^2 f_n$ is $\ell_n$-Lipschitz continuous w.r.t.\ $\|\cdot\|_2$  on the Euclidean ball $B(x_n^*, s_n)$ and satisfies an exponential growth condition
\begin{equation*}
    \left\| \nabla^2 f_n(x) - \nabla^2 f_n(x_n^*) \right\|_2^2 \leq C \exp\left( \frac{\tau_n}{2} \|x - x_n^*\|^2 \right), \quad \forall x \in \R^d \setminus B(x_n^*, s_n),
\end{equation*}
then the MFVI optimizer satisfies
\begin{equation} \label{AN-W2-bound2}
    \cW_2^2 \left( \nu_n^*, \gamma_n^* \right) \leq \frac{\ell_n^2 \left( d^2 + 2d\right)}{3 \alpha_n^4  n^2}  +  \left(\frac{nb_n}{2}\right)^{\frac{d}{2}} \frac{ C(d + 2) s_n^d}{\Gamma(d/2)\alpha_n^2(n\alpha_n - \tau_n)}  \exp \left(- \frac{(n\alpha_n - \tau_n) s_n^2}{2} \right).
\end{equation}
\end{enumerate}
\end{theorem}

We remark that $\gamma_n^*$ is the mean-field analogue of Laplace approximation \citep{Tierney1986}. \Cref{thm:non-asymptotic normal} establishes a non-asymptotic BvM theorem for MFVI in the Wasserstein distance. It applies to log-concave and log-smooth models in the form \eqref{AN-W2-bound1}. Under a relaxed smoothness condition in an $s_n$-neighborhood of $x_n^*$, together with exponential tail control, it holds in the weaker form~\eqref{AN-W2-bound2}. Our result is related to a number of previous works, in particular BvM theorems for the mean-field posterior under the classical local asymptotic normality regime \citep{Wang2019} and the total variation bounds based on Laplace approximations when $d^2/n \to 0$ \citep{Katsevich2023}; however, those works differ substantially as they either consider the finite-dimensional parametric setting or the exact (not variational) posterior. 
\begin{remark}
\begin{enumerate}[label = (\alph*)]
\item  If $d/n \to 0$ and $\alpha_n,b_n = O(1)$, the gap between $\nu_n^*$ and $\gamma_n^*$ in \eqref{AN-W2-bound1} vanishes and hence $\gamma_n^*$ can be used as a surrogate for $\nu^*_n$. The distribution $\gamma_n^*$  is much easier to compute than $\nu_n^*$, as it suffices to optimize $f_n$ for $x_n^*$ and the covariance is directly accessible. 
\item   When $\pi_n$ is Gaussian, the bound~\eqref{AN-W2-bound2} is exact. Indeed, $\nabla^2 f_n$ is $0$-Lipschitz for any $s_n > 0$, so that the right-hand side vanishes after taking $s_n \to \infty$. The bound \eqref{AN-W2-bound1} is simpler but not sharp.

\item In the setting of \eqref{AN-W2-bound1}, \cref{thm:non-asymptotic normal} yields the following bound on the squared error,
\begin{equation*}
    \left\|\EE{\nu_n^*}{X} - x_n^* \right\|^2 = \sum_{i=1}^d\left|\EE{\nu_n^*}{X_i} - x_{n,i}^* \right|^2 \leq \sum_{i=1}^d  \cW_2^2 \left( \nu_{n,i}^*, \gamma_{n,i}^* \right) = \cW_2^2 \left( \nu_{n}^*, \gamma_{n}^* \right)\leq  \frac{4 b_n^2 d}{\alpha_n^3 n}. 
\end{equation*}
If $b_n / \alpha_n^{3/2} \to 0$, the variational mean $\EE{\nu_n^*}{X}$ incurs an $o(\sqrt{d/n})$ error relative to the posterior mode $x_n^*$. In settings where $x_n^*$ is an $O(\sqrt{d/n})$-consistent estimator of some underlying truth $x_0$ (see \cite[Theorem 2.1]{He2000}), the variational estimate $\EE{\nu_n^*}{X}$ incurs no loss in statistical efficiency by achieving the same rate $\|\EE{\nu_n^*}{X} - x_0\| = O(\sqrt{d/n})$.  
\end{enumerate}
\end{remark}

\begin{proof}[Proof of \cref{thm:non-asymptotic normal}]
As $\nabla^2 f_n \succeq \alpha_n \mathrm{I}_d$, $\gamma_n^*$ is $n \alpha_n$-log-concave. Let $H_n(z) := \nabla^2 f_n(x_n^* + z)$ for $z \in \R^d$. Recall (e.g., from \citep[Section~10.1.2]{Bishop2006}) that $\gamma_n^*$ is the MFVI optimizer for the target $\tilde{\pi}_n:=\cN \left(x_n^*, (n H_n(0))^{-1} \right)$ whose potential $\tilde{V}$ satisfies $ \nabla \tilde V(x) = n H_n(0)(x- x_n^*)$. Thus \cref{thm:Lip} yields
\begin{equation}\label{eq:BvMproof0}
     \cW_2^2 \left( \nu_n^*, \gamma_n^*\right) \leq \frac{ \left\| \nabla f_n(x) -   \nabla^2 f_n(x_n^*) (x - x_n^*)\right\|_{L^2(\gamma_n^*(\rd x))}^2}{\alpha_n^2}. 
\end{equation}
Denote $r_n(x, x_n^*) :=  \nabla f_n(x) -    \nabla^2 f_n(x_n^*)(x - x_n^*) \in \R^d$, then the fundamental theorem of calculus and the fact that $\nabla f_n(x_n^*)=0$ give
\begin{equation*}
   r_n(x, x_n^*) = \int_0^1 \left( H_n(t (x - x_n^*)) -   H_n(0) \right) \left( x - x_n^*\right)  \dd t. 
\end{equation*}
By Jensen's inequality and Fubini's theorem, 
\begin{align}\label{eq:BvMproof1}
\int_{\R^d} \|r_n(x, x_n^*)\|^2 \gamma_n^*(\rd x)  &\leq \int_0^1 \int_{\R^n} \left\| \left( H_n(t (x - x_n^*)) -   H_n(0) \right) \left( x - x_n^*\right) \right \|^2 \gamma_n^*(\rd x)  \dd t \nonumber\\
&\leq \int_0^1 \int_{\R^n} \left\|H_n(t (x - x_n^*)) -   H_n(0)  \right \|_2^2  \left\|x - x_n^*  \right \|^2 \gamma_n^*(\rd x)  \dd t. 
\end{align}

(i) When $f_n$ is $b_n$-smooth, the triangle inequality yields
\begin{equation*}
    \left\|H_n(t (x - x_n^*)) -   H_n(0)  \right \|_2^2 \leq2 \left( \|H_n(t (x - x_n^*))\|_2^2 + \left\| H_n(0)  \right \|_2^2  \right) \leq 4 b_n^2. 
\end{equation*}
As $\gamma_n^* = \cN \left(x_n^*, D_n^{-1} \right)$, combining this with~\eqref{eq:BvMproof0} and~\eqref{eq:BvMproof1} yields
\begin{equation*}
     \cW_2^2 \left( \nu_n^*, \gamma_n^*\right) \leq \frac{4 b_n^2}{\alpha_n^2} \EE{\gamma_n^*}{\left\|x - x_n^*  \right \|^2 } = \frac{4 b_n^2}{\alpha_n^2} \tr(D_n^{-1}) = \frac{4 b_n^2}{n \alpha_n^2}  \sum_{j = 1}^d 1/\left(H_n(0)\right)_{jj} \leq \frac{4 b_n^2 d}{\alpha_n^3 n}. 
\end{equation*}

(ii) Under the alternative assumptions, \eqref{eq:BvMproof1} yields
\begin{align*}
 \int_{\R^d} \|r_n(x, x_n^*)\|^2 \gamma_n^*(\rd x) &\leq \underbrace{\int_0^1 t^2 \dd t \int_{B(x_n^*, s_n)} \ell_n^2 \left\|x - x_n^*  \right \|^4 \gamma_n^*(\rd x)}_{=:A} \\ 
 & \phantom{=}+ \underbrace{\int_0^1\int_{\R^d \backslash B(x_n^*, s_n)} C  \exp \left( \frac{\tau_n t^2}{2}\|x - x_n^* \|^2 \right)   \left\|x - x_n^*  \right \|^2 \gamma_n^*(\rd x)\dd t}_{=:B}. 
\end{align*}
As $\gamma_n^* = \cN \left(x_n^*, D_n^{-1} \right)$ where $D_n$ is the diagonal part of $n\nabla^2 f_n(x_n^*)$, and as $f_n$ is $\alpha_n$-log-concave and $\nabla^2 f(x_n^*)\preceq b_n\mathrm{I}_d$, we know that $n \alpha_n \mathrm{I}_d \preceq D_n \preceq n b_n \mathrm{I}_d$.  Thus, 
\begin{equation}\label{eq:B1}
\begin{aligned}
A &\leq \frac{\ell_n^2}{3} \, \EE{\gamma_n^*}{\left\|x - x_n^* \right\|^4} \\
&= \frac{\ell_n^2}{3} \sum_{i=1}^d \sum_{j=1}^d \EE{\gamma_n^*}{\left|x_i - x_{n,i}^* \right|^2 \left|x_j - x_{n,j}^* \right|^2} \\
&= \frac{\ell_n^2}{3 n^2} \left( 
    3 \sum_{i=1}^d \left(\frac{1}{(\nabla^2 f_n(x_n^*))_{ii}}\right)^2 
    + 2 \sum_{1 \leq i < j \leq d} \frac{1}{(\nabla^2 f_n(x_n^*))_{ii} (\nabla^2 f_n(x_n^*))_{jj}} 
\right) \\
&\leq \frac{(d^2 +2d )\ell_n^2}{3\alpha_n^2 n^2} 
\end{aligned}
\end{equation}
On the other hand, by writing out $\gamma_n^*$ and by $ D_n \preceq n b_n \mathrm{I}_d$, we derive that
\begin{equation}
\begin{aligned}
   B &\leq \int_{\R^d \backslash B(x_n^*, s_n)} C  \exp \left( \frac{\tau_n }{2}\|x - x_n^* \|^2 \right)   \left\|x - x_n^*  \right \|^2 \gamma_n^*(\rd x)\\
    &\leq C (2\pi)^{-\frac{d}{2}}(nb_n)^{\frac{d}{2}}  \int_{\R^d \backslash B(x_n^*, s_n)}  \exp \left( -\frac{n\alpha_n - \tau_n}{2}\|x - x_n^* \|^2 \right)   \left\|x - x_n^*  \right \|^2 \dd x.
\end{aligned}
\end{equation}
Applying \cref{lemma: incomplete-Gamma} with $\alpha = n\alpha_n-\tau_n $ and $s = s_n$, and as $ s_n > \sqrt{\frac{d + 2}{n \alpha_n - \tau_n}}$, the term $B$ satisfies
\begin{equation}\label{eq:B2 final}
\begin{aligned}
B &\leq  C \left(\frac{nb_n}{2\pi}\right)^{\frac{d}{2}} \frac{S(d) (d + 2) s_n^d}{2 (n\alpha_n-\tau_n )}  \exp \left(- \frac{(n\alpha_n-\tau_n) s_n^2}{2}  \right) \\
&= C \left(\frac{nb_n}{2\pi}\right)^{\frac{d}{2}}  \frac{(d + 2) s_n^d}{2 (n\alpha_n-\tau_n )}  \frac{2 \pi^{d/2}}{\Gamma(d/2)}  \exp \left(- \frac{(n\alpha_n-\tau_n) s_n^2}{2} \right) \\
&= C \left(\frac{nb_n}{2}\right)^{\frac{d}{2}} \frac{(d + 2) s_n^d}{\Gamma(d/2)(n\alpha_n - \tau_n)}  \exp \left(- \frac{(n\alpha_n - \tau_n) s_n^2}{2} \right).
\end{aligned}
\end{equation}
Combining~\eqref{eq:BvMproof0}, \eqref{eq:B1} and~\eqref{eq:B2 final} yields~\eqref{AN-W2-bound2}.
\end{proof}

Recall the Bayesian linear model detailed in \cref{example:linear-model}. To state a BvM theorem for this setting, we make the standard assumption that the true precision $\tau$ is known (as, e.g., in \citep[Section~6.2.1]{Miller2021}).

\begin{corollary}\label{coro:BvM for Bayesian linear model}
Let $v$ be $\cC^2$, $\alpha_0$-convex and $b_0$-smooth for some $b_0 \geq \alpha_0 \in \R$. Assume that $b_n \mathrm{I}_d \succeq \bA \succeq \alpha_n \mathrm{I}_d$ for some $\alpha_n,b_n$ with $\tau \alpha_n + \alpha_0 > 0$. Let $\nu_\tau^*$ be the MFVI optimizer for target~$\pi_{\tau}$. Set $\gamma_n^* = \cN(\beta_n^*, D_n^{-1})$ where $\beta_n^*$ is the unique solution $\beta\in\R^d$ of the equation
\begin{equation*} 
\nabla_\beta V_\Lambda(\beta) = -\tau  \bw  + \tau \bA \beta + \nabla v(\beta) = 0 \quad\mbox{where}\quad \nabla v(\beta) := \left( v^{\prime}(\beta_1), \ldots, v^{\prime}(\beta_d) \right)^\top 
\end{equation*}
and $D_n$ is the diagonal part of the matrix $\bA + \nabla^2 v(\beta_n^*)$. Then
\begin{equation*} \label{lm-CLT}
    \cW_2 \left(\nu_\tau^*, \gamma_n^* \right) \leq 2\sqrt{\frac{ d( \tau b_n + b_0)^2 }{n (\tau \alpha_n + \alpha_0)^3}}.
\end{equation*}
\end{corollary}
\Cref{coro:BvM for Bayesian linear model} addresses the non-contracting regime, where the posterior law of $\beta$ may not contract strongly around the true regression vector \citep{celentano2023mean,Mukherjee2021}. In this setting, accurate quantification of the posterior uncertainty is a difficult but important task. In our result, the approximate normality of the MFVI optimizer holds on the strength of the curvature assumption on the prior, with the condition $b_0 \geq \alpha_0 > - \tau \alpha_n$ ensuring that the prior does not overly distort the Gaussian structure of the likelihood.

\subsection{Distributed Stochastic Control}\label{sec: distributed sc}
In this application, we show how our stability result for the MFVI problem implies a stability result for distributed (a.k.a.\ decentralized) stochastic control. Fix a time horizon $T \in (0,\infty)$ and a utility function $g : \R^d \to \R$ which is $\cC^2$, concave, and $\beta$-smooth. Consider the classic problem of drift control under quadratic cost,
\begin{equation}\label{eq: stochastic control full}
 \sup_{(\alpha,X)\in\cA}  \E \left[ g(X_T) - \frac{1}{2}\int_0^T  \|\alpha(t, X_t) \|^2 \, \dd t  \right],
\end{equation}
where the function $\alpha$ controls the dynamics of the state process $(X_t)_{t\in[0,T]}$ according to the stochastic differential equation (SDE)
\begin{equation}\label{eq: SDE}
\dd X^i_t = \alpha_i(t, X_t) \dd t + \dd B^i_t, \quad X^i_0 = 0, \quad i \in [d]
\end{equation}
driven by independent Brownian motions $B^i$. This can be interpreted as a cooperative game (also known as social planner problem): each agent $i\in[d]$ controls their process $X^i$ while aiming to optimize the reward $g(X_T)$ depending on the processes of all agents. Each agent is subject to a quadratic cost and idiosyncratic noise. Popular special cases are utility functions $g$ that depend on $X_T$ only through the mean $\frac{1}{d}\sum_{i=1}^d X^i_T$, giving rise to a so-called mean-field control problem (here ``mean-field'' has a different meaning compared to the MFVI problem). See, e.g., \citep{Carmona2015,Carmona2018} and the references therein.

Before proceeding, let us detail the mathematical formalization of~\eqref{eq: stochastic control full}, which uses the so-called weak formulation to address the solvability of~\eqref{eq: SDE} while keeping $\alpha$ as general as possible. The supremum in~\eqref{eq: stochastic control full} is taken over pairs $(\alpha,X)$ where $\alpha: [0, T] \times \R^d \to \R^d$ is a measurable function and $X$ is a weak solution of the system~\eqref{eq: SDE} on a probability space $(\Omega, \mathcal{F}, (\mathcal{F}_t)_{t \in [0,T]}, \mathbb{P})$ with independent Brownian motions $B^i$ and $\int_0^T |\alpha(t, X_t)|^2 \, \dd t<\infty$ a.s. We denote by $\cA$ the set of all such $(\alpha,X)$.

The problem~\eqref{eq: stochastic control full} is coupled in the sense that each agent's control is a function of the states of all agents. We are interested in approximating~\eqref{eq: stochastic control full} by distributed controls, meaning that each agent's control is a function only of their own state. Thus, we define $\cA_{\rm dist}$ as the set of all $(\alpha,X)\in\cA$ where $\alpha = (\alpha_1, \dots, \alpha_d)$ is of the distributed form
\begin{equation}\label{eq: alpha form}
\alpha_i(t, x_1, \dots, x_d) = \hat{\alpha}_i(t, x_i), \quad i \in [d],
\end{equation}
for some measurable functions $\hat{\alpha}_i : [0, T] \times \R \to \R$ and moreover $\mathrm{Law}(X_t) \in \cP(\R)^{\otimes d}$ for all $t\in[0,T]$, meaning that the components $X^i_t$ are independent.\footnote{Independence already follows from~\eqref{eq: alpha form} and the SDE~\eqref{eq: SDE} if the SDE satisfies uniqueness in law, but uniqueness is not guaranteed for irregular $\alpha$.} The resulting control problem is 
\begin{equation}\label{eq: stochastic control distr}
\cV_{\rm dist}(g):= \sup_{(\alpha, X)\in\cA_{\rm dist}}  \E \left[ g(X_T) - \frac{1}{2}\int_0^T \sum_{i=1}^d \left|\hat\alpha_i(t, X^i_t)\right|^2 \, \dd t  \right].
\end{equation}
When the dimension~$d$ is large, distributed controls are significantly more tractable as they only depend on a one-dimensional state, hence have been of recent interest in mean-field reinforcement learning (e.g., \cite{CarmonaLauriereTan.23}) and other areas. We refer to \cite{Mahajan2012decentralized, nayyar2013decentralized} for extensive references on distributed control.

The distributed control problem~\eqref{eq: stochastic control distr} can be reformulated, via Girsanov's theorem, as the MFVI problem
\begin{equation}\label{eq: stochastic MF}
\cM(g) := \sup_{\mu \in \cP(\R)^{\otimes d}} \left(\int g \dd\mu - \kl{\mu}{\gamma_T} \right) = - \inf_{\mu \in \cP(\R)^{\otimes d}} \kl{\mu}{\pi} + \log Z
\end{equation}
where $\gamma_T:=\cN(0, T \mathrm{I}_d)$ and $\pi(\rd x) := Z^{-1} e^{g(x)} \gamma_T(\rd x)$ with $Z \in (0,\infty)$ being the normalizing constant. See \cite[Equation~(2.12)]{Lacker2024}. In particular, $\cV_{\rm dist}(g) = \cM(g)$.  As a consequence, our stability result for the MFVI problem implies stability for the distributed control problem with respect to the utility function~$g$.

\begin{corollary}\label{coro:sc reward}
     Let $g, \tilde g : \R^d \to \R$ be $\cC^2$,  concave and $\beta$-smooth. Let $\nu^*$ be the MFVI optimizer of \eqref{eq: stochastic MF} for~$g$. Then
  \begin{equation*}
  \begin{aligned}
      \left| \cV_{\rm dist}(\tilde g) - \cV_{\rm dist}(g)\right| &\leq  2\sqrt{(\beta T^2 +T)d} \|\nabla \tilde g -\nabla g\|_{L^2( \nu^*)} \\
      &+ \frac{\beta T^2 +T}{2} \|\nabla \tilde g -\nabla g\|_{L^2(\nu^*)}^2 +  \|\tilde g - g\|_{L^2( \nu^*)}.
  \end{aligned}
\end{equation*}
\end{corollary}
\begin{proof}
    Note that the potential function of $\pi$ is $V(x) = -g(x) + \frac{\|x\|^2}{2T}$. In the notation of~\eqref{eq:ELBO}, we then have
    $\int g \dd\mu - \kl{\mu}{\gamma_T}=\ELBO (\mu, V)$ and hence $\cM(g)=\cR(V)$. As $g$ is concave, $V$ is $1/T$-convex and $(\beta + 1/T)$-smooth. In view of $\cV_{\rm dist}(\cdot) = \cM(\cdot)$, the claim follows from \cref{thm:reward lip}.
\end{proof}

\section{Proof of Lipschitz Stability} \label{sec:proof-Lip}

This section is organized as follows. We begin by establishing the equivalence between \eqref{eq: MFVI} and \eqref{eq: lifted MF} in \cref{lemma:lifted KL}. %
Next, we characterize in  \cref{lemma: first-order optimality} the first-order optimality condition for \eqref{eq: lifted MF} under perturbations of the target potential. Finally, \cref{thm:Lip} is proved based on that characterization and the convexity of the potential~$V$. The stability of the optimal reward stated in \cref{thm:reward lip} then follows by direct computation and the equivalence established in \cref{lemma:lifted KL}. To show the explicit stability bounds for parameter dependence in \cref{coro:Lip}, we derive in \cref{lemma: minimizer bound} a novel density bound for the MFVI optimizer under the assumption that the target measure is log-concave and log-smooth. This bound enables explicit control of the moments of the MFVI optimizer that appear on the right-hand side of \cref{thm:Lip} and hence allow us to deduce \cref{coro:Lip}.

\subsection{Equivalence of \eqref{eq: MFVI} and \eqref{eq: lifted MF}}

Recall the lifted problem \eqref{eq: lifted MF} and the related notation introduced in \cref{se:methodology}. 

\begin{lemma}[Equivalence of \eqref{eq: MFVI} and \eqref{eq: lifted MF}]\label{lemma:lifted KL}
Recall the space $\cH_+$ from~\eqref{eq:convex cone} and the functional $\cF_V$ from~\eqref{eq: F_V}. In the setting of \cref{thm:Lip},
\begin{equation}\label{eq: lifted inf same}
    \inf_{\mu \in \cP(\R)^{\otimes d}} \kl{\mu}{\pi} = \inf_{ (T_1,\ldots, T_d)\in \cH_+ }  \cF_V(T_1,\ldots, T_d) + C
\end{equation}
with the constant $C=H(\rho) + \log Z$. The optimizers of the two problems are related as follows,
$$\nu=\bigotimes_{i=1}^d \nu_i\in \argmin_{\nu \in \cP(\R)^{\otimes d}} \kl{\nu}{\pi}$$
if and only if
\begin{equation}\label{eq:lifted KL}
    T(\nu):=(T_1(\nu), \dots, T_d(\nu))\in \argmin_{ (T_1,\ldots, T_d)\in \cH_+}  \cF_V(T_1,\ldots, T_d), 
\end{equation}
where $T_i(\nu)$ is the unique gradient of a convex function such that $T_i(\nu)_\sharp \rho_i=\nu_i$, i.e., $T_i(\nu)$ is the optimal transport from $\rho_i$ to $\nu_i$. 

In particular, there is a unique optimizer of \eqref{eq: lifted MF}, denoted $T^V=(T^V_1, \dots, T^V_d)$ and characterized by $T^V_i$ being the unique gradient of a convex function such that $(T^V_i)_\sharp \rho_i=\nu^*_i$ for the MFVI optimizer $\nu^*$.
\end{lemma}

\begin{proof}[Proof of \cref{lemma:lifted KL}]
Since $\pi(\rd x) = Z^{-1} e^{-V(x)}\dd x$ with  $\nabla^2 V \succeq \alpha \mathrm{I}_d$, \eqref{eq:exponentialGrow} and \cref{th:Lacker2024} imply that the minimizer $\nu^*$ to~\eqref{eq: MFVI} is also $\alpha$-log-concave and the density $\nu_i^*$ satisfies $|\log \nu_i^*(x_i)| \leq c_1'e^{c_2|x_i|^2}$ for some $c_1'$ sufficiently large and $c_2<\alpha/2$, for all $i\in [d]$. Therefore, $\nu^* \in \cP_{2}(\R)^{\otimes d}$ and $H(\nu^*)<\infty$. Brenier's theorem (see, e.g., \cite[Theorem~2.32]{Villani2003}, \cite[Theorem~1.2]{Gangbo1996} or \cite{Cuesta1989NotesOT}) and Caffarelli's contraction theorem  (see \cite[Theorem~11]{caffarelli2000monotonicity} and \cite[Theorem~4]{chewi2023entropic}) yield\footnote{In fact, we only use the one-dimensional special cases of those general results, which are straightforward from the explicit form of $T_i$ detailed in \cite[Section~2.1]{santambrogio2015optimal}.} that the optimal transport map $T^V= (T_1^V,\ldots, T_d^V)$ from $\rho$ to $\nu^*$ satisfies $0 < (T_i^V)' \leq \frac{1}{\sqrt{\alpha}}$, $\rho_1$-a.e., for each $i$, so that $T^V\in \cH_+$ by \cite[Proposition~1.5.2]{bogachev1998gaussian}.
Therefore, 
 \begin{equation*}
    \inf_{\nu \in \cP(\R)^{\otimes d}} \kl{\nu}{\pi}=  \inf_{\nu \in \cP_{2}(\R)^{\otimes d}} \kl{\nu}{\pi}  = \inf_{ T \in \cH_+} \kl{T_\sharp \rho}{\pi}.
\end{equation*}
For any $T \in \cH_+$ with $H(T_\sharp \rho) <\infty$, we know that  
\begin{align*}
  \kl{T_\sharp \rho}{\pi} &= \int \log \frac{\dd T_\sharp \rho}{\dd \pi}(x) T_\sharp \rho(\rd x)\\
    &= \int \log T_\sharp \rho(x) T_\sharp \rho(\rd x) + \int V(x) T_\sharp \rho(\rd x) + \log Z.
 \end{align*}
By the Monge--Ampère equation \cite[Section~4]{Villani2003}, $\det(\nabla T(u)) T_\sharp \rho (T(u)) = \rho(u)$ for all $u \in \supp(\rho) = \R^d$, whence
\begin{equation*}
    \int \log T_\sharp \rho(x) T_\sharp \rho(\rd x) = \int \log \rho(u)  \rho(\rd u) -\int \log\det(\nabla T(u))  \rho(\rd u) .
\end{equation*}
As $\nabla T(u)$ is diagonal for each $u \in \R^d$, we conclude 
\begin{align*}
       \kl{T_\sharp \rho}{\pi}  &= 
    -\int \log\det(\nabla T(u))  \rho(\rd u) + \int V(T(u))  \rho(\rd u) + \int \log \rho(u) \rho(\rd u) + \log Z\\
    &= -\int \sum_{i=1}^d \log T_i'(u_i) \rho(\rd u) + \int V(T_1(u_1), \ldots, T_d(u_d))  \rho(\rd u) \\
    &\phantom{=}+ H(\rho) + \log Z\\
    &= \cF_V(T) + C.
\end{align*}
It remains to check that
$$
    \inf_{T\in \cH_+, \, H(T_\sharp \rho) <\infty}  \cF_V(T)= \inf_{T\in \cH_+}  \cF_V(T).
$$
Indeed, let $T\in\cH_+$ be such that $H(T_\sharp \rho) = \infty$. Since $H(\rho) < \infty$, the Monge--Ampère equation implies that $-\int \sum_{i=1}^d \log T_i'(u_i)  \rho(\rd u) = \infty$.  Combining this with the fact that $V$ is bounded from below, we conclude that $\cF_V(T)=\infty$. This completes the proof of~\eqref{eq: lifted inf same}. To see the claim about the optimizers, note that if $T^V$ is an optimizer of the lifted problem, then $\nu^* := T^V_\sharp \rho$ is an optimizer of~\eqref{eq: MFVI} by~\eqref{eq: lifted inf same}, and vice versa.
\end{proof}

\subsection{Proof of \cref{thm:Lip}}

Let $V,\tilde V$ be potentials as in \cref{thm:Lip}. Fixing $T=(T_1, \dots, T_d)\in \cH_+$, we define the  linear functional $\mathcal{L}_{T,V}:\cH(\rho)\to\R$ by
\begin{equation}\label{eq: L}
    \begin{aligned}
\cH(\rho) \ni R=(R_1, \dots, R_d)\mapsto \mathcal{L}_{T,V}(R)= \sum_{i=1}^d \int \left\{\partial_i V( T(u)) R_i(u_i)- \frac{R_i'(u_i)}{(T_i)'(u_i)}   \right\} \rho(\rd u).
\end{aligned} 
\end{equation}
Thus $\mathcal{L}_{T,V}$ is the first variation of $\cF_V$ at $T$ in the direction of $R$. %

Roughly speaking, the next lemma states that the first variation of $\cF_V$ at the optimizer $T^V$ in the direction $T^{\tilde V}- T^{V}$ is zero. Since $1/T_i'$ may not belong to $\cH(\rho)$, we cannot directly take $R = T^{\tilde V}- T^{V}$; the integral of $R_i'/T_i'$ in~\eqref{eq: L} could be infinite. Instead, we use a truncation argument. Given $M>0$, we define $R^M=(R_1^M,\dots,R_d^M)$ as a truncation of $T^{\tilde V}_i-T^{V}$, 
$$R_i^M(u_i):= \begin{cases}
    T^{\tilde V}_i(u_i)-T^{V}_i(u_i) & {\rm if} \  u_i\in [-M,M], \\
      T^{\tilde V}_i(M)-T^{V}_i(M) & {\rm if} \  u_i\geq M, \\
      T^{\tilde V}_i(-M)-T^{V}_i(-M) & {\rm if} \  u_i\leq -M.
\end{cases} $$
Then Caffarelli's contraction theorem \citep{caffarelli2000monotonicity,chewi2023entropic} yields that $|(R_i^M)'(u_i)| \leq \max \{\alpha^{-1/2},{\tilde \alpha}^{-1/2}\}$ for $u_i \in [-M,M]$ while $(R_i^M)'=0$ outside $[-M,M]$. 
 We can now state the rigorous version of the first-order condition.
\begin{lemma}\label{lemma: first-order optimality}
    In the setting of \cref{thm:Lip}, we have $\mathcal{L}_{T^V,V}(R^M)=0$ for all $M >0$.
\end{lemma}

\begin{proof}%
Without loss of generality, $\alpha \leq \tilde\alpha$. Observe that
\begin{equation*}
    \int |\partial_i V(T(u))|^2  \rho(\rd u) \leq  \int \|\nabla V(x)\|^2 \nu^*(\rd x) < \infty
\end{equation*}
as $\nu^*$ is $\alpha$-log-concave and the growth condition~\eqref{eq:exponentialGrow} holds. On the other hand, $|(R_i^M)'|$ is bounded by $\frac{1}{\sqrt{\alpha}}$ on $[-M,M]$ and vanishes elsewhere. Thus, the Monge--Amp\`ere equation $(T_i^V)'(u_i)  (\nu_i^* \circ T_i^V(u_i)) = \rho_1(u_i)$, together with the fact that $\nu_i^* \in \cC^2(\R)$, yield
\begin{equation*}
  \int_\R \left|\frac{(R_i^M)'}{(T_i^V)'}\right| \dd \rho_1  \leq \frac{1}{\sqrt{\alpha}}\int_{-M}^M\frac{1}{(T_i^V)'}\dd \rho_1 = \int_{-M}^M \frac{\nu_i^*\circ T_i^{V}}{\rho_1}\dd \rho_1  = \int_{-M}^M\nu_i^*\circ T_i^{V}(u_i)\dd u_i < \infty.
\end{equation*}
Moreover, integration by parts \cite[Theorem~5.1.2]{bogachev1998gaussian} and the Monge--Amp\`ere equation yield
\begin{align*}
    \int_\R \frac{(R_i^M)'}{(T_i^V)'} \dd \rho_1 &= \int_\R {(R_i^M)'(u_i)} \frac{\nu^*_i\circ T_i^V(u_i)}{\rho_1(u_i)}  \rho_1(\rd u_i) \\
   &= \int_\R {(R_i^M)'} e^{\log(\nu^*_i\circ T_i^V)}  \dd u_i \\
    &=  - \int_\R R_i^M (\nu^*_i\circ T_i^V)(\log(\nu^*_i \circ T_i^V))'  \dd u_i\\
    &= - \int_\R R_i^M (\nu^*_i\circ T_i^V) ((\log(\nu^*_i))' \circ T_i^V ) (T_i^V)'   \dd u_i\\
    &= - \int_\R R_i^M  ((\log(\nu^*_i))' \circ T_i^V ) \dd \rho_1,\quad \forall i \in [d].
\end{align*}
Therefore, 
\begin{align}\label{eq:proof of FOC}
\mathcal{L}_{T^V,V}(R^M)= \sum_{i=1}^d \int R_i^M(u_i)\left\{\partial_i V( T^V(u))  + (\log(\nu^*_i))' \circ T_i^V(u_i) \right\} \rho(\rd u) .
\end{align}
Finally, the fixed-point equation~\eqref{eq: fixed point} gives
$$ \log(\nu^*_i)(x_i)=-\int_{\R^{d-1}} V(x_i, x_{-i}) d\nu^*_{-i}(x_{-i}) -\log\left( \int_{\R} e^{-\int V(x) d\nu^*_{-j}(x_{-i}) } d\nu^*_i(x_i)\right)$$
and thus
$$ (\log(\nu^*_i))'\circ T_i^V(u_i) = -\int_{\R^{d-1}} \partial_i V( T^V(u)) \rho_{-i}(\rd u_{-i}),\quad \forall u_i \in \R.$$
Plugging this into~\eqref{eq:proof of FOC} yields $\mathcal{L}_{T^V,V}(R^M)=0$.
\end{proof}

We can now state the proof of our first main result.

\begin{proof}[Proof of \cref{thm:Lip}]  
On the one hand,
\begin{align}
     &\phantom{=}\,\,\mathcal{L}_{T^{\tilde{V}},V}(R^M)-  \mathcal{L}_{T^V,V}(R^M)\nonumber\\
    &=  \sum_{i=1}^d \int \left\{ (\partial_i V( T^{\tilde{V}}(u))-\partial_i V( T^V(u)) )R_i^M(u_i) +\frac{R_i'(u_i)((T^{\tilde{V}}_i)'(u_i)-(T_i^V)'(u_i))}{(T_i^V)'(u_i) (T^{\tilde{V}})'(u_i)} \right\}  \rho(\rd u)\nonumber\\
    &=  \sum_{i=1}^d \int  (\partial_i V( T^{\tilde{V}}(u))- \partial_i V( T^V(u)) )R_i^M(u_i) \rho(\rd u) + \sum_{i=1}^d\int_{-M}^M\frac{((T^{\tilde{V}}_i)'(u_i)-(T_i^V)'(u_i))^2}{(T_i^V)'(u_i) (T^{\tilde{V}})'(u_i)}   \rho(\rd u)\nonumber\\
    & \geq \sum_{i=1}^d \int  (\partial_i V( T^{\tilde V}(u))- \partial_i V( T^V(u)) )R_i^M(u_i)   \rho(\rd u) \label{eq: Lip proof1}.
\end{align}
On the other hand, $ \mathcal{L}_{T^{\tilde{V}},
\tilde V}(R^M)=\mathcal{L}_{T^{V}, V}(R^M)=0$ by \cref{lemma: first-order optimality}. It follows that
\begin{align*}
   \mathcal{L}_{T^{\tilde{V}},V}(R^M)-  \mathcal{L}_{T^V,V}(R^M)&= \mathcal{L}_{T^{\tilde{V}},V}(R^M)-\mathcal{L}_{T^{\tilde{V}},
\tilde V}(R^M)\\
&= \sum_{i=1}^d \int \left\{ (\partial_i V( T^{\tilde{V}}(u))-\partial_i {\tilde V}( T^{\tilde{V}}(u)) )R^M_i(u_i) \right\}  \rho(\rd u)\\
& \leq \|R^M \|_{L^2(\rho)} \|\nabla V\circ T^{\tilde V}-\nabla \tilde V \circ T^{\tilde V} \|_{L^2(\rho)}\\
&= \|R^M \|_{L^2(\rho)} \|\nabla V-\nabla \tilde V \|_{L^2(\tilde \nu^*)}.
\end{align*}
Combining the displays, we have
\begin{align}\label{eq: Lip proof}
    \|R^M \|_{L^2(\rho)} \|\nabla V-\nabla \tilde V \|_{L^2(\tilde \nu^*)} \geq \sum_{i=1}^d \int  (\partial_i V( T^{\tilde V}(u))- \partial_i V( T^V(u)) )R_i^M(u_i)   \rho(\rd u).
\end{align}
Since
$$|R_i^M(u_i)|\leq  \begin{cases}
    |T^{\tilde V}_i(u_i)|+|T^{V}_i(u_i)| & {\rm if} \  u_i\in [-M,M], \\
      |T^{\tilde V}_i(M)|+|T^{V}_i(M)| & {\rm if} \  u_i\geq M, \\
      |T^{\tilde V}_i(-M)|+|T^{V}_i(-M)| & {\rm if} \  u_i\leq -M
\end{cases}$$
and $T^{\tilde V}_i$ and $ T^{V}_i$ are monotonically increasing,
the function  
$ |R_i^M|$  is bounded by   $$ |T^{\tilde{V}}_i|+|T^V_i| \in L^2(\rho_1).$$
Taking the limit $M\to \infty$ on both sides of~\eqref{eq: Lip proof}, dominated convergence thus implies that 
$$ \|T^{\tilde{V}}-T^V \|_{L^2(\rho)} \|\nabla V-\nabla \tilde V \|_{L^2(\tilde \nu^*)} \geq \int \langle \nabla  V( T^{\tilde{V}}(u))- \nabla V( T^V(u)) , T^{\tilde{V}}(u)-T^V(u)\rangle \rho(\rd u) .$$
Since $\nabla  V$ is $\alpha$-strongly monotone, we deduce
\begin{equation}\label{eq:ProofLipBoundalphaAppl} \|T^{\tilde{V}}-T^V \|_{L^2(\rho)} \|\nabla V-\nabla \tilde V \|_{L^2(\tilde \nu^*)} \geq \alpha \|  T^{\tilde{V}}-T^V\|^2_{L^2(\rho)}
\end{equation}
and hence
\begin{equation}\label{eq:ProofLipBound}
    \|  T^{\tilde{V}}-T^V\|_{L^2(\rho)} \leq \alpha^{-1} \|\nabla V-\nabla \tilde V \|_{L^2(\tilde \nu^*)}.
\end{equation}
\Cref{thm:Lip} follows after noting that $\|T^{\tilde{V}}-T^V\|_{L^2(\rho)}=\cW_2(\tilde{\nu}^*,\nu^*)$.
\end{proof}

\begin{remark}\label{remark:H1 bound}
    For later use, we record a strengthening of the bound~\eqref{eq:ProofLipBound} to the $\cH(\rho)$ norm. When $V$ and $\tilde V$ are both $\alpha$-convex and $\beta$-smooth, Caffarelli's contraction theorem implies that $(T_i^V)'$ and $(T_i^{\tilde V})'$ are a.e.\ bounded below by $\frac{1}{\sqrt{\beta}}$ and bounded above by $\frac{1}{\sqrt{\alpha}}$. As a consequence, the integral of $R_i'/T_i'$ with respect to $\rho$ is finite for all $R \in \cH(\rho)$, and 
    \begin{equation}\label{eq:H1 bound proof}
      \infty > \int \sum_{i=1}^d \frac{((T_i^{\tilde V})'(u_i) - (T_i^V)'(u_i))^2}{(T_i^{\tilde V})'(u_i)(T_i^V)'(u_i)} \,  \rho(\rd u) 
      \geq \alpha \|(T^{\tilde V})' - (T^V)'\|_{L^2(\rho)}^2.
    \end{equation}
    Proceeding as in the proof of \cref{thm:Lip} but using~\eqref{eq:H1 bound proof} instead of dropping the second term in the step leading to~\eqref{eq: Lip proof1}, we obtain
    \begin{equation}\label{eq:H1 bound}
             \|T^{\tilde V} - T^V\|_{\cH(\rho)} = \left(\|T^{\tilde V} - T^V\|^2_{L^2(\rho)} + \|(T^{\tilde V})' - (T^V)'\|^2_{L^2(\rho)}\right)^{1/2} \leq \alpha^{-1}\|\nabla \tilde V - \nabla V \|_{L^2(\tilde \nu^*)}.
    \end{equation}
    In particular, we record for ease of reference the bound
    \begin{equation}\label{eq:T diff derivative bound}
             \|(T^{\tilde V})' - (T^V)'\|_{L^2(\rho)}\leq \alpha^{-1}\|\nabla \tilde V - \nabla V \|_{L^2(\tilde \nu^*)}.
    \end{equation}     
\end{remark}

\begin{remark}\label{rk:beyond-log-concavity}
    In the proof of \cref{thm:Lip}, $\alpha$-log-concavity is applied to infer~\eqref{eq:ProofLipBoundalphaAppl}. This step remains valid if $\alpha$-log-concavity is relaxed to  $\alpha$-monotonicity with respect to $\rho$, meaning that
\[
\int \langle \nabla V(S(u)) - \nabla V(T^V(u)), S(u) - T^V(u) \rangle \rho(\rd u) \geq \alpha \|S - T^V\|_{L^2(\rho)}^2
\]
for all $S \in \cH_+$, where $T^V$ denotes the minimizer of~\eqref{eq: lifted MF}. This observation may suggest an avenue to go beyond the $\alpha$-log-concave setting. The above monotonicity is in the spirit of the log-Sobolev inequality for measures which is often used as an alternative to log-concavity in establishing convergence results in the literature of non-log-concave sampling~\citep{Chewi2023}.
\end{remark}

\subsection{Proof of \cref{thm:reward lip}}
Next, we prove our bound for the optimal reward function. 

\begin{proof}[Proof of \cref{thm:reward lip}]
Note that the $\beta$-log-smoothness assumption implies~\eqref{eq:exponentialGrow}. By \cref{lemma:lifted KL}, 
\begin{equation*}
\begin{aligned}
\left|\cR(\tilde V) -  \cR(V)\right| &= \left| \cF_{\tilde V}(T^{\tilde V}) - \cF_V(T^V) \right| \\
&=  \underbrace{\left|-\int \sum_{i=1}^d \left( \log ((T_i^{\tilde V})'(u_i)) - \log ((T_i^V)'(u_i)) \right)  \rho(\rd u)\right|}_{A}\\
&\quad + \underbrace{\left|\int \left( \tilde V(T^{\tilde V}(u)) - V(T^V(u)) \right)  \rho(\rd u)\right|}_{B}.
\end{aligned}
\end{equation*}
As $\tilde V,V$ are $\beta$-smooth, $(T_i^{\tilde V})', (T_i^V)' \geq 1/\sqrt{\beta}$ a.e.\ by Caffarelli's contraction theorem. We deduce
\begin{equation*}
\left| \log ((T_i^{\tilde V})'(u_i)) - \log ((T_i^V)'(u_i)) \right| \leq \sqrt{\beta} \left| (T_i^{\tilde V})'(u_i) - (T_i^V)'(u_i) \right|
\end{equation*}
and thus
\begin{equation}\label{eq: reward log}
A \leq \sqrt{\beta d} \| (T^{\tilde V})' - (T^V)' \|_{L^2(\rho)}.
\end{equation}
To bound $B$, write
\begin{equation*}
B\leq  \left| \int \tilde V(T^{\tilde V}(u)) - V(T^{\tilde V}(u)) \rho( \rd u)\right| + \left|\int V(T^{\tilde V}(u)) - V(T^V(u)) \rho( \rd u)\right|.
\end{equation*}
For the first term on the right-hand side, Jensen's inequality gives
\begin{equation}\label{eq: reward A}
   \left| \int  \tilde V(T^{\tilde V}(u)) - V(T^{\tilde V}(u))\rho( \rd u)\right| \leq \|\tilde V - V\|_{L^2(\tilde \nu^*)}.
\end{equation}
On the other hand, by $\beta$-smoothness and \citep[Theorem~5.8]{Beck2017}, 
\begin{equation*}
\left| V(T^{\tilde V}(u)) - V(T^V(u)) \right| \leq \left| \langle \nabla V(T^V(u)), T^{\tilde V}(u) - T^V(u) \rangle \right| + \frac{\beta}{2} \| T^{\tilde V}(u) - T^V(u) \|^2,\quad \forall u\in \R^d.
\end{equation*}
Since $\E_{\nu^*}\left[\Delta V(X) - \langle \nabla V(X), \nabla V (X) \rangle\right] = 0$ (see, e.g., the proof of \citep[Lemma~4.0.1]{Chewi2023}), we have
\begin{equation}\label{eq: gradient V upper bound}
\int \| \nabla V(T^V(u)) \|^2  \rho(\rd u) = \int \| \nabla V(x) \|^2 \nu^*(\rd x) = \int \Delta V(x)  \nu^*(\rd x) \leq \beta d.
\end{equation}
Therefore,
\begin{align*}
 \int \left| \langle \nabla V(T^V(u)), T^{\tilde V}(u) - T^V(u) \rangle \right|  \rho(\rd u) 
& \leq \left( \int \| \nabla V(T^V(u)) \|^2  \rho(\rd u) \right)^{1/2} \| T^{\tilde V} - T^V \|_{L^2(\rho)} \\
&\leq \sqrt{\beta d} \| T^{\tilde V} - T^V \|_{L^2(\rho)}.
\end{align*}
In summary,
\begin{equation}\label{eq: reward B}
\left| \int \left( V(T^{\tilde V}(u)) - V(T^V(u)) \right)  \rho(\rd u) \right| \leq \sqrt{\beta d} \| T^{\tilde V} - T^V \|_{L^2(\rho)} + \frac{\beta}{2} \| T^{\tilde V} - T^V \|_{L^2(\rho)}^2.
\end{equation}
Collecting \eqref{eq: reward log}, \eqref{eq: reward A}, \eqref{eq: reward B} and applying~\eqref{eq:ProofLipBound} and~\eqref{eq:T diff derivative bound}, we conclude that
\begin{equation}\label{eq:reward lip}
\begin{aligned}
       \left| \cR(\tilde V) -  \cR(V) \right|  \leq \frac{2\sqrt{\beta d}}{\alpha} \|\nabla \tilde V -\nabla V\|_{L^2(\tilde \nu^*)} + \frac{\beta}{2\alpha^2} \|\nabla \tilde V -\nabla V\|_{L^2(\tilde \nu^*)}^2 + \|\tilde V - V\|_{L^2(\tilde \nu^*)}.
\end{aligned}
\end{equation}
When $V,\tilde V$ are normalized, meaning that $\int V \dd\tilde \nu^*=\int \tilde V \dd\tilde \nu^*$, then as $\tilde \nu^*$ is $\alpha$-log-concave, applying the Poincaré inequality~\cite[Corollary~4.8.2]{bakry2014analysis} to the last term in~\eqref{eq:reward lip} yields
\begin{equation*}
       \left| \cR(\tilde V) -  \cR(V) \right| \leq \frac{2\sqrt{\beta d} +1}{\alpha}\|\nabla \tilde V - \nabla V\|_{L^2(\tilde \nu^*)} +\frac{\beta}{2\alpha^2} \|\nabla \tilde V -\nabla V\|_{L^2(\tilde \nu^*)}^2. \qedhere
\end{equation*}
\end{proof}

\subsection{Proof of \cref{coro:Lip}}

It remains to prove the explicit Lipschitz bounds.

\begin{proof}[Proof of \cref{coro:Lip}] 
The $\beta$-log-smoothness assumption implies~\eqref{eq:exponentialGrow}. 
By \cref{thm:Lip}, 
\begin{equation}\label{eq:master}
     \cW_2(\nu_{\tilde \theta}^*, \nu_\theta^*) \leq \frac{1}{\alpha_{\tilde \theta}}\|\nabla V_{\tilde \theta} - \nabla V_{\theta}\|_{L^2(\nu^*_{\theta})}.
\end{equation}
Integrating our assumption that $ \| \nabla V_{\tilde\theta}(x)-\nabla V_{\theta}(x)\| \leq L\|\tilde\theta-\theta\| f(x)$,
we further have
\begin{equation}\label{eq: curve lip}
\|\nabla V_{\tilde \theta} - \nabla V_{\theta}\|_{L^2(\nu^*_{\theta})} \leq L \|\tilde \theta- \theta \|  \, \EE{\nu^*_{\theta}}{f(X)^2}^{\frac{1}{2}}.
\end{equation}
It remains to bound $\EE{\nu^*_{\theta}}{f(X)^2}$. By \cref{lemma: minimizer bound} below, we know that
\begin{equation*}
    \nu^*_{\theta}(\rd x) \leq C_{\theta, d}\exp\left(-\frac{\alpha_{\theta}}{2}\|x-x^*_{\theta}\|^2\right)\dd x
\end{equation*}
where $x^*_{\theta} = \argmin_{x \in \R^d} V_{\theta}(x)$ and 
\begin{equation*}
       C_{\theta,d}:=   \left(\frac{2\pi}{\beta_{\theta}}\right)^{d/2}\exp\left(\frac{(\beta_{\theta}-\alpha_{\theta})d}{\alpha_{\theta}}  \left[2\kl{\cN(x^*_{\theta}, \alpha_{\theta}^{-1} \mathrm{I}_d)}{\pi_{\theta}}  + d \right] \right).
\end{equation*}
This yields the general result claimed in~\eqref{eq: curve lip general}. Moving on to (i), let $f(x) = \|x\|^{p/2}$. Then
\begin{equation}
\begin{aligned}
        \EE{\nu^*_{\theta}}{f(X)^2}=\E_{\nu^*_{\theta}}[\|X\|^p]&\leq  C_{ \theta,d}  \int \|x\|^p \exp\left(-\frac{\alpha_{\theta}}{2}\|x -x^*_{\theta}\|^2\right)\dd x. 
\end{aligned}
\end{equation}
Using the triangle inequality, the co-area formula, and $(a+b)^p\leq 2^{p-1}(a^p+b^p)$, we obtain
\begin{align*}
&\phantom{=}\,\, \int
\|x\|^p
\exp\left(-\frac{\alpha_{\theta}}{2}\|x - x^*_{\theta}\|^2\right) \dd x \\
&\leq  \int (\|x^*_{\theta}\| + \|x - x^*_{\theta}\|)^p
\exp\left(-\frac{\alpha_{\theta}}{2}\|x - x^*_{\theta}\|^2\right) \dd x \\
&= S(d)\int_0^\infty (\|x^*_{\theta}\| + r)^p
\exp\left(-\frac{\alpha_{\theta}}{2}r^2 \right) r^{d-1} \dd r \\
&\leq 2^{p - 1} S(d) \|x^*_{\theta}\|^p \int_0^\infty r^{d - 1} \exp\left(-\frac{\alpha_{\theta}}{2}r^2 \right) \dd r + 2^{p - 1} S(d)\int_0^\infty r^{p + d - 1} \exp\left(-\frac{\alpha_{\theta}}{2}r^2 \right) \dd r \\
&\leq 2^{\frac{2p+d-4}{2}} S(d) \|x^*_{\theta}\|^p \alpha_{\theta}^{-\frac{d}{2}}
\Gamma\left(\frac{d}{2}\right) + 2^{\frac{3p+d-4}{2}}S(d)\alpha_{\theta}^{-\frac{p+d}{2}}
\Gamma\left(\frac{p+d}{2}\right).
\end{align*}
This implies that
\begin{equation}\label{eq: polynomial bound}
 \EE{\nu^*_{\theta}}{f(X)^2} \leq S(d) C_{\theta,d} 2^{\frac{2p+d-4}{2}}\alpha_{\theta}^{-\frac{d}{2}} \left[\|x^*_{\theta}\|^p \Gamma\left(\frac{d}{2}\right) + 2^{\frac{p}{2}}\,\alpha_{\theta}^{-\frac{p}{2}}\,\Gamma\!\left(\frac{p+d}{2}\right) \right]. 
\end{equation}
Thus, substituting~\eqref{eq: polynomial bound} into~\eqref{eq: curve lip} yields
\begin{equation*}
  \cW_2(\nu_{\tilde \theta}^*, \nu_\theta^*) \leq \frac{L C_{\theta,d}^{1/2}}{\alpha_{\tilde \theta}}     \sqrt{ S(d)  2^{\frac{2p+d-4}{2}}\alpha_{\theta}^{-\frac{d}{2}} \left[\|x^*_{\theta}\|^p \Gamma\left(\frac{d}{2}\right) + 2^{\frac{p}{2}}\,\alpha_{\theta}^{-\frac{p}{2}}\,\Gamma\!\left(\frac{p+d}{2}\right) \right]} \|\tilde \theta- \theta \|.
\end{equation*}
Whereas for (ii), when $f(x) = \exp \left(\frac{1}{2}\|x\| \right)$,
\begin{equation}\label{eq: exponential bound}
\begin{aligned}
        \EE{\nu^*_{\theta}}{f(X)^2}&=\E_{\nu^*_{\theta}}[e^{\|X\|}] \leq e^{\|x^*_{\theta}\|}\int \exp(\|x- x^*_{\theta}\|)C_{\theta,d} \exp\left(-\frac{\alpha_{\theta}}{2}\|x -x^*_{\theta}\|^2\right)\dd x\\
       &\leq  S(d) C_{\theta,d} e^{\|x^*_{\theta}\|} \int_0^\infty e^{r- \frac{\alpha_{\theta}}{2} r^2} r^{d-1}\dd r.
\end{aligned}
\end{equation}
Therefore, we obtain
\begin{equation}
     \cW_2(\nu_{\tilde \theta}^*, \nu_\theta^*) \leq  \frac{LC_{\theta,d}^{1/2}}{\tilde  \alpha_{\theta}}  \sqrt{S(d)  e^{\|x^*_{\theta}\|} \int_0^\infty e^{r- \frac{\alpha_{\theta}}{2} r^2} r^{d-1}\dd r} \|\tilde \theta- \theta \|,
\end{equation}
completing the proof.
\end{proof}

\section{Proof of Differentiability} \label{sec:proof-differentiable}

This section is devoted to the proof of \cref{th: Differentiable}.
We begin by establishing some preliminary properties of the linear functional $\cL_{T,V}$, defined in \eqref{eq: L} and recalled here for convenience:
\begin{equation*}
    \cL_{T,V}(R)= \int \sum_{i=1}^d \left( -\frac{ R_i'(u_i)}{(T_i)'(u_i)} + \partial_i V(T_1(u_1),\ldots, T_d(u_d)) R_i(u_i)\right) \rho(\rd u),\quad \forall R\in \cH(\rho).
\end{equation*}

In contrast to the statement in \cref{lemma: first-order optimality}, the following result asserts the validity of the first-order optimality condition in \emph{any} direction of perturbation in $\cH(\rho)$. This strengthening is due to the additional smoothness assumption on the potential function.

\begin{lemma}\label{lemma: first-order optimality diff}
    In the setting of \cref{th: Differentiable}, we have $\frac{1}{\sqrt{\beta}}\leq (T_i^\theta)' \leq \frac{1}{\sqrt{\alpha}}$ a.e. for all $i\in [d]$ and $\mathcal{L}_{T^\theta,V_\theta}(R)=0$ for all $R\in \CH(\rho)$. 
\end{lemma}

\begin{proof}[Proof of \cref{lemma: first-order optimality diff}]
The first claim follows from Caffarelli's contraction theorem. As $\alpha \mathrm{I}_d \preceq \nabla^2 V_\theta \preceq \beta \mathrm{I}_d$, \eqref{eq: gradient V upper bound} gives
\begin{equation*}
    \int |\partial_i V_\theta(T^\theta(u))|^2  \rho(\rd u) \leq  \int \|\nabla V_\theta(x)\|^2 \nu^*(\rd x) \leq \beta d.
\end{equation*}
Together with the first claim, this shows the boundedness of the linear functional~$\cL_{T^\theta, V_{\theta}}$. The proof of $\mathcal{L}_{T^\theta,V_\theta}(R)=0$ then follows along the lines of \cref{lemma: first-order optimality}: we first establish the result for $R = (R_1,\ldots, R_d)\in (\cC_c^\infty(\R))^d$ and then extend to $R\in \cH(\rho)$ by an approximation argument.
\end{proof}

Next, we derive a bilinear relation that will be the foundation for the proof of \cref{th: Differentiable}. \Cref{lemma: first-order optimality diff} shows that for any $\theta,\theta_0\in \Theta$ and $R\in \cH(\rho)$, 
\begin{equation}\label{eq: BR precursor}
\mathcal{L}_{T^{\theta},V_{\theta_0}}(R)- \mathcal{L}_{T^{\theta_0}, V_{\theta_0}}(R) =\mathcal{L}_{T^{\theta},V_{\theta_0}}(R)- \mathcal{L}_{T^{\theta}, V_{\theta}}(R).
\end{equation}
Fix $j\in [d]$ and $u\in \R^d$. Applying the mean-value theorem to the scalar function $ [0,1]\ni \lambda \mapsto  \partial_j V_{\theta_0}\left(\lambda T^\theta(u) + (1-\lambda)T^{\theta_0}(u)\right)$ yields an intermediate point $\lambda^j(u) \in[0,1]$ such that 
\begin{equation}\label{eq: bar T}
    \bar T^{j,\theta}(u):=\lambda^j(u) T^\theta(u) + (1-\lambda^j(u))T^{\theta_0}(u)
\end{equation}    
satisfies 
\begin{equation}\label{eq: intermed point form}
    \partial_j V_{\theta_0}(T^\theta(u))- \partial_j V_{\theta_0}(T^{\theta_0}(u)) =  \sum_{i=1}^d \partial_{ij} V_{\theta_0}(\bar T^{j,\theta}(u)) (T_i^{\theta}(u_i)- T_i^{\theta_0}(u_i)).
\end{equation}
Noting that the possible choices of $\lambda^j(u)\in[0,1]$ form the zero set of a continuous function of $(\lambda,u)$, the measurable selection theorem \citep[Theorem~18.17]{Aliprantis2006} allows us to choose $\lambda^j(u)$ such that $u\mapsto\lambda^j(u)$ and thus $u\mapsto \bar T^{j,\theta}(u)$ is measurable, ensuring that the integrals below are well-defined.
Using~\eqref{eq: intermed point form}, the left-hand side of \eqref{eq: BR precursor} becomes
\begin{equation}\label{eq:developmentL2}   
    \begin{aligned}
    &\mathcal{L}_{T^{\theta},V_{\theta_0}}(R)- \mathcal{L}_{T^{\theta_0}, V_{\theta_0}}(R) \\ 
    & =\int \sum_{j=1}^d [\partial_j V_{\theta_0}(T^{\theta}(u))-\partial_j V_{\theta_0}(T^{\theta_0}(u))] R_j(u_j) 
    + \sum_{i=1}^d \frac{ R_i'(u_i)}{(T^{\theta_0}_i)'(u_i)} - \frac{ R_i'(u_i)}{(T^{\theta}_i)'(u_i)}  \rho(\rd u)\\
    &=     \int  \sum_{i,j=1}^d \partial_{ij} V_{\theta_0}(\bar T^{j,\theta}(u)) (T_i^{\theta}(u_i)- T_i^{\theta_0}(u_i)) R_j(u_j) + \sum_{i=1}^d \frac{R_i'(u_i)((T^{\theta}_i)'(u_i)-(T^{\theta_0}_i)'(u_i))}{(T^{\theta_0}_i)'(u_i) (T^{\theta}_i)'(u_i)}  \rho(\rd u).
\end{aligned}
\end{equation}
On the other hand, the right-hand side of~\eqref{eq: BR precursor} can be written as
\begin{equation}\label{eq:developmentR2}
    \begin{aligned}
\mathcal{L}_{T^{\theta},V_{\theta_0}}(R)- \mathcal{L}_{T^{\theta}, V_{\theta}}(R)   
   &= \sum_{i=1}^d  \int \left\{ (\partial_i V_{\theta_0}( T^{\theta}(u))-\partial_i V_{\theta}( T^{\theta}(u)) )R_i(u_i) \right\}  \rho(\rd u) \\
    &=  \int \inner{\nabla V_{\theta_0}( T^{\theta}(u))-\nabla V_{\theta}( T^{\theta}(u))}{R(u)}  \rho(\rd u).
\end{aligned}
\end{equation}
Set $S := T^{\theta} - T^{\theta_0}$ for brevity. Combining~\eqref{eq: BR precursor},~\eqref{eq:developmentL2} and~\eqref{eq:developmentR2}, we arrive at the bilinear relation
\begin{equation}\tag{\textrm{BR}} \label{eq: bilinear equation}
        \begin{aligned}
         &   \int  \sum_{i,j=1}^d \partial_{ij} V_{\theta_0}(\bar T^{j,\theta}(u)) S_i(u_i) R_j(u_j) + \sum_{i=1}^d \frac{S_i'(u_i)R_i'(u_i)}{(T^{\theta_0}_i)'(u_i) (T^{\theta}_i)'(u_i)}  \rho(\rd u) \\
         =&   \int \inner{\nabla V_{\theta_0}( T^{\theta}(u))-\nabla V_{\theta}( T^{\theta}(u))}{R(u)} \rho(\rd u),\quad \forall R\in \cH(\rho).
    \end{aligned}
\end{equation}
This relation is fundamental to all the subsequent arguments. Roughly speaking, \cref{th: Differentiable} will be proved by dividing~\eqref{eq: bilinear equation} by $\theta-\theta_0$ and showing that the resulting left-hand side converges to the  bilinear operator $\cB_{\theta_0}$ as $\theta\to\theta_0$. As a preparation, we first show the following a priori bound on the difference of two optimizers of~\eqref{eq: lifted MF}. This bound is also of independent interest. 

\begin{proposition}\label{pr: Lp estimate}
    Let $p\geq2$. In the setting of \cref{th: Differentiable}, there is a constant $C_{\alpha,\beta, d, p}$ such that 
    \begin{equation*}
        \|T^\theta - T^{ \theta_0}\|_{L^p(\rho)}+  \|(T^{\theta})' - (T^{\theta_0})'\|_{L^p(\rho)}\leq C_{\alpha,\beta, d, p}\|\nabla V_{\theta}(T^{\theta}) - \nabla V_{\theta_0}(T^{\theta})\|_{L^p(\rho)}.
    \end{equation*}
    More specifically, we have
    \begin{align}\label{eq: lp bound}
        \cW_p(\nu^*_\theta, \nu^*_{\theta_0}) = \|T^\theta - T^{ \theta_0}\|_{L^p(\rho)} \leq d^{\frac{p-2}{2}} \frac{\alpha + \sqrt{d}\beta}{\alpha^2}\|\nabla V_{\theta}(T^{\theta}) - \nabla V_{\theta_0}(T^{\theta})\|_{L^{p}(\rho)}
    \end{align}
    and
    \begin{align}\label{eq: lp bound derivative}
        \|(T^{\theta})' - (T^{\theta_0})'\|_{L^p(\rho)} \leq  
        M_{p/(p-1)}^{\frac{p-1}{p}} \frac{ d^{\frac{p-2}{2}} +  d^{p-1}\kappa +  d^{\frac{2p-1}{2}}\kappa^2}{\alpha} \|\nabla V_{\theta}(T^{\theta}) - \nabla V_{\theta_0}(T^{\theta})\|_{L^{p}(\rho)},
    \end{align}
    where $\kappa=\beta/\alpha$ and $M_r = \left(\frac{\pi}{2}\right)^r\E_{\rho_1}|X|^r$ for $r\geq 1$.
\end{proposition}

\begin{remark}\label{rk: Gaussian Poincare inequality}
The constant $M_r$ in \Cref{pr: Lp estimate} arises from the (generalized) Gaussian Poincaré inequality~\cite[Corollary~1.7.3]{bogachev1998gaussian} which states that
\[
\int_\R \left|f - \int_\R f \dd\rho_1\right|^r \dd \rho_1 \leq M_r \int_\R |f'|^r \dd \rho_1
\]
for any function $f \in L^r(\rho_1)$ that admits a weak derivative $f' \in L^r(\rho_1)$.
\end{remark}

\begin{remark}\label{rk: Lp tilde norm}
   For some of the calculations below, it will be convenient to use the following norm for a vector-valued function $f = (f_1,\ldots, f_d):\R^d\to \R^d$,
\begin{equation}\label{eq: tilde lp}
    \|f\|_{\tilde L^p(\rho)}  := \left(\int \sum_{i=1}^d |f_i(u)|^p  \rho(\rd u)\right)^{1/p}.
\end{equation}
This norm satisfies $\|f\|_{\tilde L^p(\rho)}^p  = \sum_{i=1}^d \|f_i\|_{L^p(\rho)}^p$. It differs from $\|f\|_{L^p(\rho)}$ by using the $p$-norm on the image space $\R^d$ instead of the Euclidean 2-norm, but is of course equivalent. We note that 
\[
\int |\inner{f}{g}|\dd\rho \leq \|f\|_{\tilde L^p(\rho)}\|g\|_{\tilde L^q(\rho)}
\]
whenever $1/p + 1/q =1$, by using H\"older's inequality both on $\R^d$ and for the integrals.
Furthermore,  
\begin{enumerate}[label = (\alph*)]
    \item if $p \geq 2$, then $\|f\|_{\tilde L^p(\rho)} \leq \|f\|_{ L^p(\rho)} \leq  d^\frac{p-2}{2p}\|f\|_{\tilde L^p(\rho)}$;
    \item if $p\in (0,2]$, then $d^{\frac{p-2}{2p}}\|f\|_{\tilde L^p(\rho)} \leq \|f\|_{L^p(\rho)} \leq \|f\|_{\tilde L^p(\rho)}$;
    \item for $p =2$, we have $\|f\|_{L^2(\rho)} = \|f\|_{\tilde L^2(\rho)}$.
\end{enumerate}
Here (a) follows from the fact that $\sum_{i=1}^d a_i^p \leq (\sum_{i=1}^d a_i^2)^{p/2} \leq d^{\frac{p-2}{2}}\sum_{i=1}^d a_i^p$ for $a_1,\ldots, a_d \in \R_+$ when $p\geq 2$ while (b) follows from $d^{\frac{p-2}{2}}\sum_{i=1}^d a_i^p \leq (\sum_{i=1}^d a_i^2)^{p/2} \leq \sum_{i=1}^d a_i^p$ for $a_1,\ldots, a_d \in \R_+$ when $p\in (0,2]$. 
\end{remark}

\begin{proof}[Proof of \cref{pr: Lp estimate}]
Set $S_i = T_i^{\theta} - T_i^{\theta_0}$. We consider the test functions $R_i = |S_i|^{p-2}S_i$ for $i\in [d]$. Note that $t\mapsto |t|^{p-2}t$ is $\cC^1$ as $p>2$. Using the regularity of $T_i^{\theta},T_i^{\theta_0}$ (\cref{lemma: first-order optimality diff}), we see that $R = (R_1,\ldots, R_d) \in \cH(\rho)$ and its components have the derivatives
\begin{equation*}
     R_i' = \begin{cases}
    (p-1)(S_i)^{p-2}S_i',& S_i\geq0,\\
    (p-1)(-S_i)^{p-2}S_i',& S_i\leq 0.
\end{cases}
\end{equation*}
For brevity, we set $a_{ij}(u) = \partial_{ij}^2 V_{\theta_0}( \bar{T}^{j,\theta}(u))$ for $i,j\in [d]$ for the remainder of the proof.

\emph{Step 1: Proof of \eqref{eq: lp bound}.} Substituting $R$ into~\eqref{eq: bilinear equation}, the first term on the left-hand side of~\eqref{eq: bilinear equation} becomes
\begin{align*}
   &\phantom{=}\,\,\int  \sum_{i,j=1}^d \partial_{ij} V_{\theta_0}(\bar T^{j,\theta}(u)) S_i(u_i)R_j(u_j)\rho(\rd u)\\
    &= \sum_{1\leq i\leq d} \int a_{ii}(u) |S_i|^p(u_i) \rho(\rd u) + \sum_{1\leq i\neq j\leq d} \int a_{ij}(u) S_i(u_i)(|S_j|^{p-2}S_j)(u_j) \rho(\rd u).
\end{align*}
Since $\alpha \mathrm{I}_d \preceq \nabla^2 V \preceq \beta \mathrm{I}_d$, we know that $\max_{ij} |a_{ij}(u)| \leq \beta$. Thus Fubini's theorem implies
\begin{align*}
    \left| \int a_{ij}(u) S_i(u_i)(|S_j|^{p-2}S_j)(u_j) \rho(\rd u) \right|
    &\leq \beta  \int |S_i(u_i)| |S_j|^{p-1}(u_j) \rho(\rd u) \\
    &=\beta \int_\R  |S_i(u_i)| \rho_1(\rd u_i) \int_\R |S_j|^{p-1}(u_j)\rho_1(\rd u_j)\\
   &=  \beta \|S_i\|_{L^1(\rho_1)}\|S_j\|_{L^{p-1}(\rho_1)}^{p-1}.
\end{align*}
To further bound the last expression, recall from \cref{thm:Lip} that
$$\|S\|_{L^1(\rho)} \leq \|S\|_{L^2(\rho)} \leq \frac{1}{\alpha}\|\nabla V_{\theta}(T^{\theta}) - \nabla V_{\theta_0}(T^{\theta})\|_{L^2(\rho)} \leq \frac{1}{\alpha} \|\nabla V_{\theta}(T^{\theta}) - \nabla V_{\theta_0}(T^{\theta})\|_{L^p(\rho)}.$$
Therefore, 
\begin{align*}
    \sum_{1\leq i\neq j\leq d}\|S_i\|_{L^1(\rho_1)}\|S_j\|_{L^{p-1}(\rho_1)}^{p-1} &= \sum_{j=1}^d \|S_j\|_{L^{p-1}(\rho_1)}^{p-1} \sum_{1\leq i\leq d, i\neq j}\|S_i\|_{L^1(\rho_1)}\\
   &\leq  \|S\|_{\tilde L^{p-1}(\rho)}^{p-1} \|S\|_{\tilde L^1(\rho)}\\
   &\leq  \sqrt{d}  \|S\|_{ L^1(\rho)} \|S\|_{L^{p-1}(\rho)}^{p-1}\\
   &\leq   \frac{\sqrt{d}}{\alpha}\|\nabla V_{\theta}(T^{\theta}) - \nabla V_{\theta_0}(T^{\theta})\|_{L^p(\rho)}\|S\|_{L^p(\rho)}^{p-1},
\end{align*}
where $\|\cdot\|_{\tilde L^p(\rho)}$ is defined in~\eqref{eq: tilde lp} and we used that $\|S\|_{\tilde L^1(\rho)} \leq \sqrt{d} \|S\|_{ L^1(\rho)}$ by \cref{rk: Lp tilde norm}. Combining the displays and using the fact that $ \|S\|_{\tilde L^p(\rho)}^p \geq d^{\frac{2-p}{2}}  \|S\|_{ L^p(\rho)}^p$ for $p\geq 2$ by \cref{rk: Lp tilde norm}, we conclude that
\begin{equation}\label{eq: Lp lhs}
    \begin{aligned}
     &\phantom{=}\,\,\int  \sum_{i,j=1}^d \partial_{ij} V_{\theta_0}(\bar T^{j,\theta}(u)) S_i(u_i)R_j(u_j)\rho(\rd u) \\
     &\geq \alpha \sum_{i=1}^d \int |S_i|^p(u_i) \rho(\rd u)  - \beta \sum_{1\leq i\neq j\leq d}\|S_i\|_{L^1(\rho_1)}\|S_j\|_{L^{p-1}(\rho_1)}^{p-1}\\
     & \geq \alpha \|S\|_{\tilde L^p(\rho)}^p  -  \sqrt{d}\frac{\beta}{\alpha}\|\nabla V_{\theta}(T^{\theta}) - \nabla V_{\theta_0}(T^{\theta})\|_{L^p(\rho)}\|S\|_{L^p(\rho)}^{p-1}\\
     & \geq  \alpha d^{\frac{2-p}{2}} \|S\|_{ L^p(\rho)}^p  - \sqrt{d}\frac{\beta}{\alpha}\|\nabla V_{\theta}(T^{\theta}) - \nabla V_{\theta_0}(T^{\theta})\|_{L^p(\rho)}\|S\|_{L^p(\rho)}^{p-1}.
\end{aligned}
\end{equation}
On the other hand, the second term on the left-hand side of~\eqref{eq: bilinear equation} is non-negative as 
\begin{align*}
 \int \frac{S_i'(u_i)R_i'(u_i)}{(T^{\theta_0}_i)'(u_i) (T^{\theta}_i)'(u_i)} \rho(\rd u)  &= \int_{\{S_i\geq0 \}} (p-1)\frac{(S_i')^2(u_i)(S_i)^{p-2}(u_i)}{(T^{\theta_0}_i)'(u_i) (T^{\theta}_i)'(u_i)} \rho_1(\rd u_i)\\
 &\phantom{=}+ \int_{\{S_i<0\}} (p-1)\frac{(S_i')^2(u_i)(-S_i)^{p-2}(u_i)}{(T^{\theta_0}_i)'(u_i) (T^{\theta}_i)'(u_i)} \rho_1(\rd u_i)\geq 0.
\end{align*}
Moreover, using H\"older's inequality for $\|\cdot\|_{\tilde L^p(\rho)}$ followed by $\|\cdot\|_{\tilde L^p(\rho)} \leq \|\cdot\|_{L^p(\rho)}$ for $p\geq 2$ (\cref{rk: Lp tilde norm}), the right-hand side of~\eqref{eq: bilinear equation} satisfies
\begin{align}\label{eq:Lp rhs}
     &\int \inner{\nabla V_{\theta_0}( T^{\theta}(u))-\nabla V_{\theta}( T^{\theta}(u))}{R(u)} \rho(\rd u)\leq \|\nabla V_{\theta}(T^{\theta}) - \nabla V_{\theta_0}(T^{\theta})\|_{L^{p}(\rho)} \|S\|_{L^p(\rho)}^{p-1}.
\end{align}
Combining~\eqref{eq: bilinear equation} with the bounds~\eqref{eq: Lp lhs} and~\eqref{eq:Lp rhs} for the left and right-hand sides of~\eqref{eq: bilinear equation},
\begin{equation*}
    \alpha d^{\frac{2-p}{2}} \|S\|_{L^p(\rho)}^p \leq \|\nabla V_{\theta}(T^{\theta}) - \nabla V_{\theta_0}(T^\theta)\|_{L^{p}(\rho)} \|S\|_{L^p(\rho)}^{p-1} + \sqrt{d} \frac{\beta}{\alpha}\|\nabla V_{\theta}(T^{\theta}) - \nabla V_{\theta_0}(T^{\theta})\|_{L^p(\rho)}\|S\|_{L^p(\rho)}^{p-1}.
\end{equation*}
This can be rearranged into the desired inequality
\begin{equation*}
    \|S\|_{L^p(\rho)} \leq d^{\frac{p-2}{2}} \frac{\alpha + \sqrt{d}\beta}{\alpha^2}\|\nabla V_{\theta}(T^{\theta}) - \nabla V_{\theta_0}(T^{\theta})\|_{L^{p}(\rho)},
\end{equation*}
completing the proof of \eqref{eq: lp bound}.

\emph{Step 2: Proof of \eqref{eq: lp bound derivative}.} We turn to the derivative, $\|S'\|_{L^p(\rho)}$. Consider the test functions 
\[
R_i(u_i) = \int_0^{u_i} (|S_i'|^{p-2}S_i')(v_i)\dd v_i - \int_\R \int_0^{u_i} (|S_i'|^{p-2}S_i')(v_i)\dd v_i \rho_1(\rd u_i),
\]
then $R= (R_1,\ldots, R_d) \in \cH(\rho)$. Note that $\int R_i \dd \rho_1 = 0$ and 
\begin{equation}\label{eq: proof step 2 R prime}
R_i' = |S_i'|^{p-2}S_i'.
\end{equation}
Hence, the Gaussian Poincar\'e inequality 
(see \cref{rk: Gaussian Poincare inequality}) gives
\begin{equation}\label{eq: proof step 2 poincare}
\int_\R |R_i|^{p/(p-1)}\dd \rho_1 \leq  M_{p/(p-1)} \int_\R |R_i'|^{p/(p-1)} \dd \rho_1 =  M_{p/(p-1)} \int_\R |S_i'|^p \dd \rho_1 = M_{p/(p-1)}\|S_i'\|_{L^p(\rho_1)}^p.
\end{equation}
Next, we substitute $R$ into~\eqref{eq: bilinear equation}. 
Using Cauchy--Schwarz for the inner product given by $(a_{ij})_{ij} \preceq  d\beta \mathrm{I}_d$, followed by H\"older's inequality,  the fact that $\|\cdot\|_{\tilde L^{p}(\rho)} \leq \|\cdot\|_{ L^{p}(\rho)}$ due to $p\geq2$, then~\eqref{eq: proof step 2 poincare}, and finally $\|\cdot\|_{L^{p/(p-1)}(\rho)} \leq \|\cdot\|_{\tilde L^{p/(p-1)}(\rho)}$ due to $\frac{p}{p-1}\in (1,2]$, we obtain
\begin{equation}\label{eq:Lp' lhs 1}
    \begin{aligned}
    \left| \int  \sum_{i,j=1}^d \partial_{ij} V_{\theta_0}(\bar T^{j,\theta}(u)) S_i(u_i)R_j(u_j)\rho(\rd u)\right|&\leq d\beta \int \|S(u)\|\|R(u)\|\rho(\rd u)\\
    & \leq d\beta \|S\|_{L^p(\rho)} \|R\|_{L^{p/(p-1)}(\rho)} \\
       &\leq d\beta \|S\|_{L^p(\rho)} \|R\|_{\tilde L^{p/(p-1)}(\rho)} \\
       & \leq d\beta M_{p/(p-1)}^{\frac{p-1}{p}} \|S\|_{L^p(\rho)}  \|S'\|_{\tilde L^p(\rho)}^{p-1}\\
       & \leq d\beta M_{p/(p-1)}^{\frac{p-1}{p}} \|S\|_{L^p(\rho)}  \|S'\|_{L^p(\rho)}^{p-1}.
\end{aligned}
\end{equation}
In addition, \cref{lemma: first-order optimality diff},~\eqref{eq: proof step 2 R prime} and \cref{rk: Lp tilde norm} yield
\begin{equation}\label{eq:Lp' lhs 2}
    \sum_{i=1}^d\int \frac{S_i'(u_i)R_i'(u_i)}{(T^{\theta_0}_i)'(u_i) (T^{\theta}_i)'(u_i)} \rho(\rd u) \geq \alpha \sum_{i=1}^d \|S_i'\|_{L^p(\rho_1)}^p \geq \alpha d^{\frac{2-p}{2}} \|S'\|^p_{L^p(\rho)}.
\end{equation}
Finally, we also note from \cref{rk: Lp tilde norm} and~\eqref{eq: proof step 2 poincare} that
\begin{equation}\label{eq:Lp' rhs}
    \begin{aligned}
        &\phantom{=}\,\,\int \inner{\nabla V_{\theta_0}( T^{\theta}(u))-\nabla V_{\theta}( T^{\theta}(u))}{R(u)} \rho(\rd u) \\
        &\leq   M_{p/(p-1)}^{\frac{p-1}{p}} \|\nabla V_{\theta}(T^{\theta}) - \nabla V_{\theta_0}(T^{\theta})\|_{\tilde L^{p}(\rho)} \|S'\|_{\tilde L^p(\rho)}^{p-1}\\
     &\leq M_{p/(p-1)}^{\frac{p-1}{p}} \|\nabla V_{\theta}(T^{\theta}) - \nabla V_{\theta_0}(T^{\theta})\|_{L^{p}(\rho)} \|S'\|_{L^p(\rho)}^{p-1}.
    \end{aligned}
\end{equation}
Combining~\eqref{eq:Lp' lhs 2} with~\eqref{eq: bilinear equation}, \eqref{eq:Lp' lhs 1} and~\eqref{eq:Lp' rhs}, we get
\begin{equation*}
     \alpha d^{\frac{2-p}{2}} \|S'\|_{L^p(\rho)}^p \leq M_{p/(p-1)}^{\frac{p-1}{p}} \|\nabla V_{\theta}(T^{\theta}) - \nabla V_{\theta_0}(T^{\theta})\|_{L^{p}(\rho)} \|S'\|_{L^p(\rho)}^{p-1} +  d\beta M_{p/(p-1)}^{\frac{p-1}{p}} \|S\|_{L^p(\rho)}  \|S'\|_{L^p(\rho)}^{p-1}.
\end{equation*}
Together with the already established bound~\eqref{eq: lp bound} on~$S$, this yields
\begin{equation*}
\begin{aligned}
         \|S'\|_{L^p(\rho)} &\leq  \frac{M_{p/(p-1)}^{\frac{p-1}{p}} d^{\frac{p-2}{2}}}{\alpha} \left(1 + \beta d^{\frac{p-2}{2} + 1} \frac{\alpha + \sqrt{d}\beta}{\alpha^2}\right)\|\nabla V_{\theta}(T^{\theta}) - \nabla V_{\theta_0}(T^{\theta})\|_{L^{p}(\rho)}\\
         & =  M_{p/(p-1)}^{\frac{p-1}{p}} \frac{ d^{\frac{p-2}{2}} +  d^{p-1}\kappa +  d^{\frac{2p-1}{2}}\kappa^2}{\alpha} \|\nabla V_{\theta}(T^{\theta}) - \nabla V_{\theta_0}(T^{\theta})\|_{L^{p}(\rho)}
\end{aligned}
\end{equation*}
where $\kappa := \beta/\alpha$.
\end{proof}

The following lemma collects auxiliary results for the proof of \cref{th: Differentiable}.

\begin{lemma}\label{lemma:portmanteau}
In the setting of \cref{th: Differentiable}, we have $\|T^\theta - T^{\theta_0}\|_{\cH(\rho)} \leq \frac{1}{\alpha}\|\nabla V_{\theta}-\nabla  V_{\theta_0} \|_{L^2(\nu_{\theta_0}^*)}$ for $\theta,\theta_0\in\Theta$. Moreover, the following limits hold as $\theta \to \theta_0$,
    \begin{enumerate}[label=(\roman*)]
        \item $T^{\theta} \to T^{\theta_0}$ in $\cH(\rho)$, and in particular in measure $\rho$;
        \item $\frac{1}{(T^\theta)'} \to \frac{1}{(T^{\theta_0})'}$ in $L^q(\rho)$, for all $q >0$;
        \item $\partial_{ij} V_{\theta_0}(\bar T^{j,\theta}) \to \partial_{ij} V_{\theta_0}(T^{\theta_0})$  in $L^q(\rho)$, for all $q >0$ and $i,j \in [d]$. 
    \end{enumerate}
\end{lemma}

\begin{proof}
The first inequality was shown in~\eqref{eq:H1 bound}. Item~(i) follows from that inequality, \cref{lemma: L^2} and the assumption that   $ \| \nabla  V_{\theta}- \nabla V_{\theta_0} \|_{L^p(\cN(0, (1+\vae)\alpha^{-1}{\rm I}_d))}\leq  C_{\theta_0} |\theta-\theta_0|$. Item~(ii) follows from~(i), the fact that $\frac{1}{\sqrt{\beta}}\leq (T_i^\theta)'$ a.e.\ for all $i\in [d]$ (see \cref{lemma: first-order optimality diff}), and the bounded convergence theorem. For~(iii), recall from~\eqref{eq: bar T} that $\bar T^{j,\theta}(u)$ lies on the line segment $[T^{\theta}(u),T^{\theta_0}(u)]$ by construction. Hence~(i) implies that $\bar T^{j,\theta}\to T^{\theta_0}$  in measure $\rho$, and then~(iii) follows from the bounded convergence theorem as $\partial_{ij} V_{\theta_0}$ is continuous and bounded according to $V_{\theta_0}\preceq \beta \mathrm{I}_d$.
\end{proof}
Next, we prove \cref{th: Differentiable} by showing that in the limit $\theta\to\theta_0$, the right-hand side of~\eqref{eq: bilinear equation}, divided by $\theta-\theta_0$, converges to the  bilinear operator $\cB_{\theta_0}$ that characterizes the derivative in~\eqref{eq:derivative-formula}. 

\begin{proof}[Proof of \cref{th: Differentiable}.]
Recall that the bilinear operator $\cB_{\theta_0}: \cH(\rho) \times \cH(\rho) \to \R$ is defined as
\begin{equation*}
   \cB_{\theta_0}(S,R) = \int \langle [\nabla^2 V_{\theta_0}(T^{\theta_0}(u))] S(u),  R(u) \rangle \rho(\rd u)+ \sum_{i=1}^d\int \frac{S_i'(u_i)R_i'(u_i)}{((T^{\theta_0}_i)'(u_i))^2} \rho(\rd u).
\end{equation*}
In view of $\alpha \mathrm{I}_d \preceq \nabla^2 V_{\theta_0}\preceq \beta \mathrm{I}_d$ and \cref{lemma: first-order optimality diff}, $\cB_{\theta_0}$ is continuous and $\alpha$-coercive, i.e.,
\begin{equation}\label{eq: bilin coercive}
\cB_{\theta_0}(S,S) \geq \alpha \|S\|_{\cH(\rho)}^2.
\end{equation}
The Lax--Milgram theorem \cite[Section~6.2.1]{evans2022partial} thus implies that $S \mapsto \cB_{\theta_0}(S,\cdot)$ is a bijection from $\cH(\rho)$ to $\cH(\rho)$. 
In particular, we can define $\xi_{\theta_0} \in \cH(\rho)$ as the unique solution to
\begin{equation}\label{eq: BR limit}
    \cB_{\theta_0}(\xi_{\theta_0}, R) = - \int \inner{ \partial_\theta \nabla V_{\theta_0} \circ T^{\theta_0}}{R} \rho(\rd u), \quad \forall R\in \cH(\rho).
\end{equation}
Here we have used that $\partial_\theta \nabla V_{\theta_0} \circ T^{\theta_0}\in \cH(\rho)$ due to \cref{lemma: first-order optimality diff} and the assumption that $x\mapsto \partial_\theta \nabla V_{\theta_0}(x)$ is Lipschitz.

To prove the theorem, we need to establish the following convergence,
\begin{equation}\label{eq: Diff proof goal}
 \lim_{\theta \to \theta_0}  \left\|\frac{T^{\theta} - T^{\theta_0}}{\theta - \theta_0} -\xi_{\theta_0} \right\|_{\cH(\rho)}=0.
\end{equation}

\emph{Step 1.} We have
\begin{equation}\label{eq: proof bilin to zero}
    \begin{aligned}
             &\phantom{=}\,\,\mathcal{L}_{T^{\theta}, V_{\theta_0}}(R) - \mathcal{L}_{T^{\theta},V_{\theta}}(R) -   \int \inner{ (\theta_0-\theta)\partial_\theta \nabla V_{\theta_0} \circ T^{\theta_0}(u)}{R(u)} \rho(\rd u) \\
      &= \int \inner{\nabla V_{\theta_0}(T^{\theta}(u))-\nabla V_{\theta}(T^{\theta}(u)) - (\theta_0 - \theta) \partial_\theta \nabla V_{\theta_0}\circ T^{\theta_0}(u)}{R(u)} \rho(\rd u)\\
     &\leq  \|\nabla V_\theta \circ T^{\theta}- \nabla V_{\theta_0}\circ T^\theta - (\theta-\theta_0) \partial_\theta \nabla V_{\theta_0} \circ T^{\theta_0}\|_{L^2(\rho)}\|R\|_{L^2(\rho)}.
    \end{aligned}
\end{equation}
For the first norm on the right-hand side, note that 
\begin{align}\label{eq: proof bilin to zero2}
& \phantom{=}\,\,\|\nabla V_\theta \circ T^{\theta}- \nabla V_{\theta_0}\circ T^\theta - (\theta-\theta_0) \partial_\theta \nabla V_{\theta_0} \circ T^{\theta_0}\|_{L^2(\rho)}\\
     &\leq  \|\nabla V_\theta- \nabla V_{\theta_0} -  (\theta-\theta_0) \partial_\theta \nabla V_{\theta_0} \|_{L^2(\nu_{\theta}^*)} + |\theta-\theta_0|\| \partial_\theta \nabla V_{\theta_0} \circ T^{\theta}  -  \partial_\theta \nabla V_{\theta_0} \circ T^{\theta_0}\|_{L^2(\rho)}. \nonumber
\end{align}
As $\theta \mapsto \nabla V_\theta$ is Fr\'echet differentiable at $\theta_0 \in \Theta$ in $L^2(\cN(0,(1+\vae)\alpha^{-1}\mathrm{I}_d))$,
\begin{equation}\label{eq: diff rhs 1}
    \frac{1}{|\theta - \theta_0|}\left\|\nabla V_{\theta} - \nabla V_{\theta_0} - (\theta - \theta_0)\partial_\theta \nabla V_{\theta_0}\right\|_{L^2(\cN(0,(1+\vae)\alpha^{-1}\mathrm{I}_d))} \to 0.
\end{equation}
By \cref{lemma: L^2}, we know that there exists $C_\theta >0$ such that 
\[
 \|\nabla V_{\theta_0} - \nabla V_{\theta} - (\theta_0 - \theta)\partial_\theta \nabla V_{\theta_0}\|_{L^2(\nu_{\theta}^*)} \leq C_\theta \left\|\nabla V_{\theta} - \nabla V_{\theta_0} - (\theta - \theta_0)\partial_\theta \nabla V_{\theta_0}\right\|_{L^2(\cN(0,(1+\vae)\alpha^{-1}\mathrm{I}_d))}.
\]
As $\limsup_{\theta \to \theta_0} C_\theta \leq M_{\theta_0}$ for some constant $M_{\theta_0}$ by \cref{lemma: uniform L2 control}, combining this with \eqref{eq: diff rhs 1} yields
\begin{equation}\label{eq:diff o theta 1}
    \|\nabla V_{\theta_0} - \nabla V_{\theta} - (\theta_0 - \theta)\partial_\theta \nabla V_{\theta_0}\|_{L^2(\nu_{\theta}^*)} = o(|\theta-\theta_0|).
\end{equation} 
Moreover, since $x\mapsto \partial_\theta V_{\theta_0}(x)$ is Lipschitz,  \cref{lemma:portmanteau} implies that
\begin{equation}\label{eq:diff o theta 2}
 \| \partial_\theta \nabla V_{\theta_0} \circ T^{\theta}  -  \partial_\theta \nabla V_{\theta_0} \circ T^{\theta_0}\|_{L^2(\rho)}\to0.
\end{equation}
Combining~\eqref{eq:diff o theta 1} and \eqref{eq:diff o theta 2} with~\eqref{eq: proof bilin to zero} and~\eqref{eq: proof bilin to zero2} yields 
\begin{equation}\label{eq:right-hand side-developement}
    \left|\mathcal{L}_{T^{\theta}, V_{\theta_0}}(R) - \mathcal{L}_{T^{\theta},V_{\theta}}(R) -   \int \inner{ (\theta_0-\theta)\partial_\theta \nabla V_{\theta_0} \circ T^{\theta_0}(u)}{R(u)} \rho(\rd u) \right|\leq   o(|\theta - \theta_0|) \|R\|_{L^2(\rho)}.
\end{equation}
Using~\eqref{eq: BR precursor},~\eqref{eq: BR limit} and~\eqref{eq:right-hand side-developement}, we have established that
\begin{equation}\label{eq:BR second term}
    \begin{aligned}
        &\left|\frac{ \mathcal{L}_{T^{\theta}, V_{\theta_0}}(R)-  \mathcal{L}_{T^{\theta_0},V_{\theta_0}}(R)}{\theta-\theta_0} - \cB_{\theta_0}(\xi_{\theta_0}, R)\right|\\
        =&  \left|\frac{ \mathcal{L}_{T^{\theta}, V_{\theta_0}}(R)-  \mathcal{L}_{T^{\theta},V_{\theta}}(R)}{\theta_0-\theta} - \int \inner{ \partial_\theta \nabla V_{\theta_0} \circ T^{\theta_0}}{R} \rho(\rd u)\right|\\
        =& o(1)\|R\|_{\cH(\rho)}.
    \end{aligned}
\end{equation}

\emph{Step 2.} Set $S :=T^{\theta}- T^{\theta_0} \in \cH(\rho)$. Using~\eqref{eq:developmentL2}, we have
\begin{align*}
&\mathcal{L}_{T^{\theta},V_{\theta_0}}(R) - 
     \mathcal{L}_{T^{\theta_0}, V_{\theta_0}}(R) - \cB_{\theta_0}(S,R) \\=&  \underbrace{\int \sum_{i,j=1}^d \left(\partial_{ij}V_{\theta_0}(\bar T^{j,\theta}(u)) -\partial_{ij} V_{\theta_0}(T^{\theta_0}(u))\right)S_i(u_i)R_j(u_j) \rho(\rd u)}_{=:A}\\
     &+ \sum_{i=1}^d \underbrace{\int  \frac{S_i'(u_i)R_i'(u_i)}{(T^{\theta_0}_i)'(u_i)}\left(\frac{1}{(T^{\theta}_i)'(u_i)} - \frac{1}{(T^{\theta_0}_i)'(u_i)}\right) \rho(\rd u)}_{=:B}
\end{align*}
 for any $R\in \cH(\rho)$. By H\"older's inequality, for each $i,j \in [d]$,
 \begin{align*}
     A_{ij}:=&\left|\int  \left(\partial_{ij}V_{\theta_0}(\bar T^{j,\theta}(u)) -\partial_{ij} V_{\theta_0}(T^{\theta_0}(u))\right)S_i(u_i)R_j(u_j) \rho(\rd u)\right| \\
     \leq& \|\partial_{ij}V_{\theta_0}\circ \bar T^{j,\theta} -\partial_{ij} V_{\theta_0}\circ T^{\theta_0}\|_{L^{\frac{2p}{p-2}}(\rho)}\|S_i\|_{L^{p}(\rho)}\|R_j\|_{L^2(\rho)}.
 \end{align*}
Using \cref{pr: Lp estimate},
$$
  \|S_i\|_{L^{p}(\rho)}\leq \|S\|_{L^{p}(\rho)} 
  \leq C_{\alpha,\beta, d, p}\|\nabla V_{\theta}(T^{\theta}) - \nabla V_{\theta_0}(T^{\theta})\|_{L^p(\rho)}
  = C_{\alpha,\beta, d, p}\|\nabla V_{\theta} - \nabla V_{\theta_0}\|_{L^p(\nu^*_\theta)}.
$$
With \cref{lemma: L^2}, \cref{lemma: uniform L2 control}, and the assumption that $\|\nabla V_{\theta} - \nabla V_{\theta_0}\|_{L^p(\cN(0, (1+\vae)\alpha^{-1}\mathrm{I}_d))} \leq C_{\theta_0}|\theta - \theta_0|$, we conclude that
\begin{equation}\label{eq:diff Aij}
    A_{ij} \leq  C_{\alpha,\beta, d, p, \theta_0,\vae} \|\partial_{ij}V_{\theta_0}\circ \bar T^{j,\theta} -\partial_{ij} V_{\theta_0}\circ T^{\theta_0}\|_{L^{\frac{2p}{p-2}}(\rho)} |\theta - \theta_0| \|R\|_{L^2(\rho)} = o(|\theta - \theta_0|) \|R\|_{L^2(\rho)},
\end{equation}
where the last equality follows from \cref{lemma:portmanteau}. Therefore, 
\begin{equation}\label{eq:diff A}
    |A| \leq \sum_{i,j=1}^d A_{ij} = o(|\theta - \theta_0|) \|R\|_{L^2(\rho)} \leq o(|\theta - \theta_0|) \|R\|_{\cH(\rho)}.
\end{equation}
In addition, by \cref{lemma: first-order optimality diff}, H\"older's inequality and \cref{lemma:portmanteau}, we have
\begin{equation}\label{eq:diff B}
    \begin{aligned}
    |B|&\leq  \sqrt{\beta} \sum_{i=1}^d \|S_i'\|_{L^{p}(\rho)} \left\|\frac{1}{(T^{\theta}_i)'} - \frac{1}{(T^{\theta_0}_i)'}\right\|_{L^{\frac{2p}{p-2}}(\rho)}\|R_i'\|_{L^2(\rho)} \leq o(|\theta -\theta_0|)\|R\|_{\cH(\rho)}.
\end{aligned}
\end{equation}
Combining~\eqref{eq:diff A} and~\eqref{eq:diff B}, we have shown that
\begin{equation*}
     \left|\mathcal{L}_{T^{\theta}, V_{\theta_0}}(R)-  \mathcal{L}_{T^{\theta_0},V_{\theta_0}}(R) - \cB_{\theta_0}(S,R)\right|  = o(|\theta-\theta_0|)\|R\|_{\CH(\rho)},
\end{equation*}
that is,
\begin{equation}\label{eq:frechet-development}
\left|\frac{ \mathcal{L}_{T^{\theta}, V_{\theta_0}}(R)-  \mathcal{L}_{T^{\theta_0},V_{\theta_0}}(R)}{\theta-\theta_0} - \cB_{\theta_0}(S/(\theta -\theta_0),R)\right|  = o(1)\|R\|_{\cH(\rho)}.
\end{equation}

\emph{Step 3.} We can now complete the proof. Using the triangle inequality followed by~\eqref{eq:frechet-development} and~\eqref{eq:BR second term}, we have 
\begin{equation*}
\begin{aligned}
   &\phantom{=}\,\,\left| \cB_{\theta_0}(S/(\theta-\theta_0)-\xi_{\theta_0},R)\right| \\
   & = \left| \cB_{\theta_0}(S/(\theta-\theta_0),R) - \cB_{\theta_0}(\xi_{\theta_0}, R)\right| \\
   & \leq \left| \cB_{\theta_0}(S/(\theta-\theta_0),R) -  \frac{ \mathcal{L}_{T^{\theta}, V_{\theta_0}}(R)-  \mathcal{L}_{T^{\theta_0},V_{\theta_0}}(R)}{\theta-\theta_0} \right| + \left|\frac{ \mathcal{L}_{T^{\theta}, V_{\theta_0}}(R)-  \mathcal{L}_{T^{\theta_0},V_{\theta_0}}(R)}{\theta-\theta_0} - \cB_{\theta_0}(\xi_{\theta_0}, R)\right|\\
   &= o(1)\|R\|_{\cH(\rho)}
\end{aligned}
\end{equation*}
for all $R \in \cH(\rho)$.
In particular, this holds for $R = R_\theta:=S/(\theta-\theta_0) - \xi_{\theta_0}$. As $\cB_{\theta_0}$ is $\alpha$-coercive, we deduce
\begin{equation*}
\begin{aligned}
      \alpha \|R_\theta\|_{\cH(\rho)}^2  \leq \left| \cB_{\theta_0}(R_\theta,R_\theta) \right|  \leq o(1)\left\|R_\theta\right\|_{\cH(\rho)}
\end{aligned}
\end{equation*}
and thus $\|R_\theta\|_{\cH(\rho)}=o(1)$. This is~\eqref{eq: Diff proof goal} and completes the proof of  \cref{th: Differentiable}.
\end{proof}

\section{Technical Lemmas}\label{se:technicalLemmas}

In this section, we provide several properties of $\alpha$-log-concave measures and their MFVI optimizer that were used in the main proofs. The last result of the section, \cref{lemma: incomplete-Gamma}, provides a calculation used in the proof of \cref{thm:non-asymptotic normal}.

\begin{lemma}\label{lemma: minimizer bound}
Let $\pi\in\cP(\R^d)$ be  $\cC^2$, $\alpha$-log-concave and $\beta$-log-smooth with potential function~$V$. Let $\nu^*$ be the MFVI optimizer for target $\pi$ and  $x^* := \argmin_{x\in \R^d} V(x)$. Then
\begin{equation}\label{eq: minimizer bound1}
 \nu^*(\rd x) \leq  C\exp\left(-\frac{\alpha}{2}\|x-x^*\|^2\right)\dd x,
\end{equation}
where
\begin{equation}\label{eq:C}
    C:=    \left(\frac{2\pi}{\beta}\right)^{d/2}\exp\left(\frac{(\beta-\alpha)d}{\alpha}  \left[2 \kl{\cN(x^*, \alpha^{-1} \mathrm{I}_d)}{\pi}  +  d \right]\right).
\end{equation}
Moreover,
\begin{equation}\label{eq: second moment bound}
     \int \|x- x^*\|^2\dd\nu^* \leq \frac{4\kl{\mu}{\pi} + 2d}{\alpha}, \quad \forall \mu \in \cP(\R)^{\otimes d}.
\end{equation}
\end{lemma}

\begin{remark}
 The stated constant $C$ in~\eqref{eq:C} depends on $x^*$. This dependence can be eliminated by replacing $\cN(x^*, \alpha^{-1} \mathrm{I}_d)$ with $\cN(0, \alpha^{-1} \mathrm{I}_d)$. In fact, the proof shows that the bound~\eqref{eq: minimizer bound1} remains valid if $\cN(x^*, \alpha^{-1} \mathrm{I}_d)$ is replaced by any $\mu\in\cP(\R)^{\otimes d}$.
\end{remark}

\begin{proof}[Proof of \cref{lemma: minimizer bound}]
    Without loss of generality, we may assume that $x^* = 0$. 
    Then by strong convexity and smoothness,
    \begin{equation*}
         V(0) + \frac{\alpha}{2}\|x\|^2 \leq V(x)\leq V(0) + \frac{\beta}{2}\|x\|^2.
    \end{equation*}
    From~\eqref{eq: fixed point}, we know that
    \begin{equation*}
        \nu_i^*(x_i)= Z_i \exp\left(-\int_{\R^{d-1}} V(x_i,x_{-i})\nu_{-i}^*(\rd x_{-i})\right),
    \end{equation*}
    where $Z_i = \int_\R \exp\left(-\int_{\R^{d-1}} V(x_i,x_{-i})\nu_{-i}^*(\rd x_{-i})\right) \dd x_i$. Thus,
    \begin{align*}
        \nu_i^*(x_i)
       &\leq  \frac{\exp\left(- \frac{\alpha}{2}\int_{\R^{d-1}} \|x\|^2\nu_{-i}^*(\rd x_{-i})\right)}{\int_\R \exp\left(- \frac{\beta}{2}\int_{\R^{d-1}} \|x\|^2\nu_{-i}^*(\rd x_{-i})\right)\dd x_i }\\
        &  =\frac{\exp\left(-\frac{\alpha}{2}x_i^2\right)}{\int_\R \exp\left(-\frac{\beta}{2}x_i^2\right)\dd x_i}\frac{\exp\left(- \frac{\alpha}{2}\int_{\R^{d-1}} \|x_{-i}\|^2\nu_{-i}^*(\rd x_{-i})\right)}{\exp\left(- \frac{\beta}{2}\int_{\R^{d-1}} \|x_{-i}\|^2\nu_{-i}^*(\rd x_{-i})\right) }\\
        &= \exp\left(\frac{\beta-\alpha}{2}\int_{\R^{d-1}} \|x_{-i}\|^2\nu_{-i}^*(\rd x_{-i})\right) \sqrt{\frac{2\pi}{\beta}}\exp\left(-\frac{\alpha}{2}x_i^2\right)\\
       &\leq   \exp\left(\frac{\beta-\alpha}{2}\int_{\R^{d-1}} \|x\|^2\nu^*(\rd x)\right) \sqrt{\frac{2\pi}{\beta}}\exp\left(-\frac{\alpha}{2}x_i^2\right).
    \end{align*}
     As $\nu\in \cP(\R)^{\otimes d}$, we conclude
    \begin{equation}\label{eq: minimizer bound implicit}
        \nu^*(x) \leq   \left(\frac{2\pi}{\beta}\right)^{d/2} \exp\left(\frac{(\beta-\alpha)d}{2}\int \|x\|^2\nu^*(\rd x)\right)\exp\left(-\frac{\alpha}{2}\|x\|^2\right).
    \end{equation}
    It remains to bound the second moment, $\int \|x\|^2\nu^*(\rd x)$. Using the triangle inequality, Talagrand's transportation inequality \cite[Corollary~9.3.2]{bakry2014analysis}, the optimality of~$\nu$, and finally~\cite[Proposition~1]{Durmus2019}, we have for any $\mu \in \cP(\R)^{\otimes d}$ the bound
    \begin{align*}
        \int \|x\|^2d\nu^* &= \cW^2_2(\nu^*, \delta_0) \\
       &\leq  2\left(\cW^2_2(\nu^*, \pi) + \cW^2_2(\pi, \delta_0)\right)\\
       &\leq  \frac{4}{\alpha} \kl{\nu^*}{\pi} + 2\int \|x\|^2 \dd \pi\\
       &\leq  \frac{4}{\alpha} \kl{\mu}{\pi} + 2\int \|x\|^2 \dd \pi \\
       &\leq  \frac{4}{\alpha} \kl{\mu}{\pi} + \frac{2d}{\alpha}.
    \end{align*}
    This is the claim~\eqref{eq: second moment bound}. Inserting that result into~\eqref{eq: minimizer bound implicit} and specializing to $\mu=\cN(x^*, \alpha^{-1} \mathrm{I}_d)$ gives
     \begin{equation*}
        \nu^*(x) \leq   \left(\frac{2\pi}{\beta}\right)^{d/2}\exp\left(\frac{(\beta-\alpha)d}{\alpha}  \left[2 \kl{\cN(x^*, \alpha^{-1} \mathrm{I}_d)}{\pi}  +  d \right]\right) \exp\left(-\frac{\alpha}{2}\|x\|^2\right),
    \end{equation*}
    which is the claim~\eqref{eq: minimizer bound1}.
\end{proof}

\begin{lemma}\label{lemma: L^2}
Let $\pi\in\cP(\R^d)$ be  $\cC^2$, $\alpha$-log-concave and $\beta$-log-smooth with potential function~$V$. Let $\nu^*$ be the MFVI optimizer for target $\pi$ and  $x^* := \argmin_{x\in \R^d} V(x)$. We have 
   \[
   \|f\|_{L^2(\nu^*)} \leq C_{\alpha,\beta,d, x^*, \vae }\|f\|_{L^2(\cN(0,(1+\vae)\alpha^{-1}\mathrm{I}_d))}
   \]
   for any $\vae>0$ and any measurable function $f:\R^d\to\R$, where 
   \begin{equation}\label{eq: lemma L^2 constant}
          C_{\alpha,\beta,d, x^*, \vae } =  \left(\frac{2\pi}{\beta}\right)^{d/2} \exp\left(\frac{(\beta-\alpha)d}{\alpha}  \left[2 \kl{\cN(x^*, \alpha^{-1} \mathrm{I}_d)}{\pi}  +  d \right] + \frac{(1+\vae^{-1})\alpha}{2(1+\vae)}\|x^*\|^2\right).
   \end{equation}
\end{lemma}

\begin{proof}
    Using \cref{lemma: minimizer bound} and the fact that $\|x-x^*\|^2 \geq \frac{1}{1+\vae}\|x\|^2 - \frac{1+\vae^{-1}}{1+\vae} \|x^*\|^2$ for any $x\in \R^d$ and $\vae>0$ (which follows from Young's inequality),
    \[
     \nu^*(\rd x) \leq  C\exp\left(-\frac{\alpha}{2}\|x-x^*\|^2\right)\dd x \leq C\exp\left(\frac{(1+\vae^{-1})\alpha}{2(1+\vae)}\|x^*\|^2\right)\exp\left(-\frac{\alpha}{2(1+\vae)}\|x\|^2\right)\dd x,
    \]
    where $C$ is specified in~\eqref{eq:C}. Thus 
    \begin{align*}
        \int |f(x)|^2 \dd\nu^*(x)&\leq  \int |f(x)|^2 C\exp\left(-\frac{\alpha}{2(1+\vae)}\|x-x^*\|^2\right)\dd x\\
       &\leq  C_{\alpha, \beta, d, x^*, \vae}\int |f(x)|^2 \exp\left(-\frac{\alpha}{2(1+\vae)}\|x\|^2\right)\dd x
    \end{align*}
    as desired.
\end{proof}

The following fact is standard.

\begin{lemma}\label{lemma: convergence of minimizer}
    Let $f, f_n: \R^d \to \R$ be $\alpha$-strongly convex. Let $x_n^*$ (resp.\ $x^*$) be the minimizer of~$f_n$ (resp.\ $f$). If $f_n\to f$ pointwise, then $x_n^* \to x^*$ and $f_n(x_n^*) \to f(x^*)$. 
\end{lemma}

\begin{proof}
    Pointwise convergence of finite convex functions implies Gamma convergence \cite[Example~5.13]{DalMaso.93}. Moreover, uniform strong convexity and pointwise convergence together imply equi-coercivity. For a Gamma convergent sequence of equi-coercive functions with unique minimizers, the minimizers and the minimum values converge \cite[Corollary 7.24]{DalMaso.93}.
\end{proof}

\begin{lemma}\label{lemma: uniform L2 control}
     Let $\{\pi_\theta\}_{\theta \in \Theta}$ be a family of $\cC^2$, $\alpha$-log-concave and $\beta$-log-smooth probability measures with potential functions $\{V_\theta\}_{\theta \in \Theta}$ and $x^*_\theta := \argmin_{x\in \R^d} V_\theta(x)$. Denote by $C_{\alpha,\beta,d, x_\theta^*, \vae }$ the constant specified in \cref{lemma: L^2} corresponding to $\theta$. For any $\theta_0\in\Theta$ and $\vae>0$, there exists a constant $M_{\alpha,\beta,d, x_{\theta_0}^*, \vae }$ such that  if $ V_\theta \to  V_{\theta_0}$ pointwise as $\theta\to \theta_0$, then
\[
\limsup_{\theta \to \theta_0} C_{\alpha,\beta,d, x_\theta^*, \vae } \leq M_{\alpha,\beta,d, x_{\theta_0}^*, \vae }.
\]
\end{lemma}
\begin{proof}
    Let $ V_\theta \to  V_{\theta_0}$ pointwise as $\theta\to \theta_0$, then $x_\theta^* \to x_{\theta_0}^*$ and $V_{\theta}(x_\theta^* ) \to V_{\theta_0}(x_{\theta_0}^*)$ by \cref{lemma: convergence of minimizer}. Since $V_\theta$ is $\alpha$-convex and $\beta$-smooth,
    \begin{align*}
      \frac{\alpha}{2}\|x-x_\theta^*\|^2  \leq V_\theta(x) - V_\theta(x_\theta^*) \leq  \frac{\beta}{2}\|x-x_\theta^*\|^2.
    \end{align*}
    Thus, we have
    \[
    \E_{\cN(x_\theta^*, \alpha^{-1} \mathrm{I}_d)}[V_\theta(X)]  \leq V_\theta(x_\theta^*) + \frac{\beta}{2}\E_{\cN(x_\theta^*, \alpha^{-1} \mathrm{I}_d)}\|X-x_\theta^*\|^2 =  V_\theta(x_\theta^*) + \frac{d\beta}{2\alpha}
    \]
    and
    \[
    \int e^{-V_\theta(x)}\dd x \leq e^{-V_\theta(x_\theta^*)} \int e^{-\frac{\alpha}{2}\|x-x_\theta^*\|^2 }\dd x = C_{\alpha,d}e^{-V_\theta(x_\theta^*)}.
    \]
    In view of $V_{\theta}(x_\theta^* ) \to V_{\theta_0}(x_{\theta_0}^*)$,  it follows that
      \[
    \limsup_{\theta \to \theta_0} \kl{\cN(x_\theta^*, \alpha^{-1} \mathrm{I}_d)}{\pi_\theta}  < \infty.
    \]
    In view of~\eqref{eq: lemma L^2 constant} and  $x_\theta^* \to x_{\theta_0}^*$, this shows the claim.
\end{proof}

Finally, we state an identity that was used in the proof of \cref{thm:non-asymptotic normal}.

\begin{lemma}\label{lemma: incomplete-Gamma}
Let $\alpha > 0$ and $s > \sqrt{(d+2)/\alpha}$. Then 
\begin{equation*}
     I(\alpha,s):=  \int_{\{y \in \R^d: \|y\|\ge s\}} \|y\|^2 \,\exp\Bigl(-\frac{\alpha}{2}\,\|y\|^2\Bigr)\dd y 
     \, \leq\,  S(d) \frac{(d + 2) s^d}{2 \alpha}  \exp \left(- \frac{\alpha s^2}{2}  \right), 
\end{equation*}
where \(S(d)\) is the surface area of the unit sphere in \(\R^d\). 
\end{lemma}

\begin{proof}
By passing to spherical coordinates, 
\begin{equation*}
   I(\alpha, s) = S(d) \int_s^\infty r^{2+d-1}\,\exp\Bigl(-\frac{\alpha}{2} \,r^2\Bigr)\dd r. 
\end{equation*}
The change of variable $u=\frac{\alpha}{2}\,r^2$ gives
\begin{equation*}
    r^{d+1}\dd r = \left(\sqrt{\frac{2u}{\alpha}}\right)^{d+1}\frac{\dd u}{\alpha\sqrt{\frac{2u}{\alpha}}}
=\frac{1}{\alpha}\left(\frac{2u}{\alpha}\right)^{\frac{d}{2}}\dd u
\end{equation*}
and thus
\begin{equation*}
    I(\alpha,s)= S(d)\,\frac{1}{\alpha}\left(\frac{2}{\alpha}\right)^{\frac{d}{2}}
\int_{\frac{\alpha s^2}{2}}^\infty u^{\frac{d}{2}}\,\exp(-u) \dd u. 
\end{equation*}
By \cite[Theorem~4.4.3]{Gabcke1979}, as $\alpha s^2 > d + 2$, the incomplete gamma function is bounded by
\begin{equation*}
    \int_{\frac{\alpha s^2}{2}}^\infty u^{\frac{d}{2}}\,\exp(-u)\dd u \leq \frac{d + 2}{2}  \left(\frac{\alpha s^2}{2} \right)^{\frac{d}{2}} \exp \left(- \frac{\alpha s^2}{2}  \right)
\end{equation*}
and the claim follows.
\end{proof}

\bibliography{main_cleaned}
\appendix

\end{document}